\documentclass[10pt,a4paper]{amsart}

\usepackage{amsmath}
\usepackage{amsthm}
\usepackage{amsfonts}
\usepackage{amssymb}
\usepackage{hyperref}
\usepackage{tikz}

\newtheorem{theorem}{Theorem}[section]
\newtheorem{lemma}[theorem]{Lemma}
\newtheorem{proposition}[theorem]{Proposition}
\newtheorem{corollary}[theorem]{Corollary}
\newtheorem{definition}[theorem]{Definition}

\theoremstyle{remark}
\newtheorem{remark}[theorem]{Remark}

\numberwithin{equation}{section} 

\newcommand*{\N}{\ensuremath{\mathbb{N}}}
\newcommand*{\Z}{\ensuremath{\mathbb{Z}}}
\newcommand*{\R}{\ensuremath{\mathbb{R}}}
\newcommand*{\C}{\ensuremath{\mathbb{C}}}

\newcommand*{\Per}{\ensuremath{\mathbf{P}}}
\newcommand*{\Tra}{\ensuremath{\mathbf{T}}}
\newcommand*{\Sca}{\ensuremath{\mathbf{S}}}

\DeclareMathOperator{\BV}{\operatorname{BV}}

\definecolor{mycolor1}{RGB}{3, 165, 39} %green
\definecolor{mycolor2}{RGB}{31, 75, 168} %blue
\definecolor{mycolor3}{RGB}{252, 152, 22} %orange
\definecolor{mycolor4}{RGB}{182, 26, 10} %red

\title{Fourier series windowed by a bump function}
\author{Paul Bergold}
\address{(Paul Bergold) Zentrum Mathematik, Technische Universit\"at M\"unchen, Germany}
\email{bergold@ma.tum.de}
\author{Caroline Lasser}
\address{(Caroline Lasser) Zentrum Mathematik, Technische Universit\"at M\"unchen, Germany}
\email{classer@ma.tum.de}
\date{\today}

\begin{document}
\maketitle

%%%%%%%%%%%%%%%%%%%%%%
\begin{abstract}
We study the Fourier transform windowed by a bump function.
We transfer Jackson's classical results on the convergence of the Fourier series of a periodic function to windowed series of a not necessarily periodic function.
Numerical experiments illustrate the obtained theoretical results.
\end{abstract}
%%%%%%%%%%%%%%%%%%%%%%
%%%%%%%%%%%%%%%%%%%%%%
%%%%%%%%%%%%%%%%%%%%%%
\section{Introduction}\label{sec:introduction}
The theory of Fourier series plays an essential role in numerous applications of contemporary mathematics.
It allows us to represent a periodic function in terms of complex exponentials.
Indeed, any
\textcolor{black}{
square integrable function $f\colon\R\to\C$ of period $2\pi$ has a norm-convergent Fourier series such that (see e.g. \cite[Prop.~4.2.3.]{Butzer:1971aa})}
\begin{align*}
	f(x) 
	=\sum_{k=-\infty}^\infty \widehat f(k) e^{ikx}
	\quad\text{\textcolor{black}{almost everywhere,}}
\end{align*}
where the Fourier coefficients are defined according to
\begin{align*}
	\widehat f(k) 
	:=\frac{1}{2\pi}\int_{-\pi}^\pi f(x)e^{-ikx}\ \mathrm{d}x,
	\quad k\in\Z.
\end{align*}
By the classical results of Jackson in 1930, see \cite{Jackson:1994aa}, the decay rate of the Fourier coefficients and therefore the convergence speed of the Fourier series depend on the regularity of the function.
If $f$ has a jump discontinuity, then the order of magnitude of the coefficients is $\mathcal{O}(1/|k|)$, as $|k|\to \infty$.
Moreover, if $f$ is a smooth function of period $2\pi$, say $f\in C^{s+1}(\R)$ for some $s\ge 1$, then the order improves to $\mathcal{O}(1/|k|^{s+1})$.\\

In the present paper we focus on the reconstruction of a not necessarily periodic function with respect to a finite interval $(-\lambda,\lambda)$.
For this purpose let us think of a smooth, non-periodic real function $\psi\colon\R\to\R$, which we want to represent by a Fourier series in $(-\lambda,\lambda)$.
Therefore, we will examine its $2\lambda$-periodic extension, see Figure~\ref{fig:periodization}.
Whenever $\psi(-\lambda^+)\ne \psi(\lambda^-)$, the periodization has a jump discontinuity at $\lambda$, and thus the Fourier coefficients are $\mathcal{O}(1/|k|)$.
An easy way to eliminate these discontinuities at the boundary, is to multiply the original function by a smooth window, compactly supported in $[-\lambda,\lambda]$.
The resulting periodization has no jumps.
Consequently, one expects faster convergence of the windowed Fourier sums.\\

%%%%%%%%%%%%%%%%%%%%%%%%%%%%%%%%%%%%%%%%%%
%figure
\begin{figure}
	\includegraphics{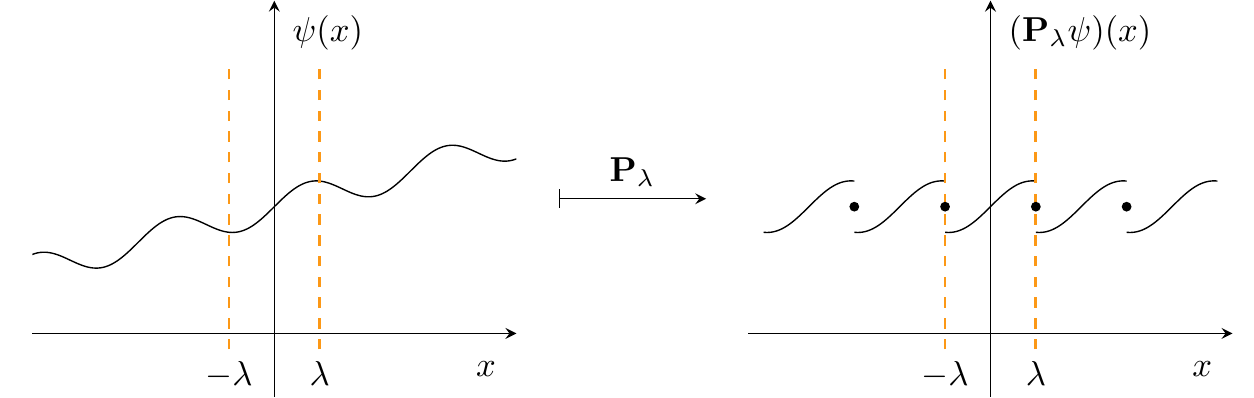}
	\caption{Effect of the periodization:
		If $\psi(-\lambda^+)\ne \psi(\lambda^-)$, then the $2\lambda$-periodic extension produces jump discontinuities at $\pm\lambda$.
		Consequently, the order of the Fourier coefficients is $\mathcal{O}(1/|k|)$.
		}
	\label{fig:periodization}	
\end{figure}
%%%%%%%%%%%%%%%%%%%%%%%%%%%%%%%%%%%%%%%%%%

The concept of windowed Fourier atoms has been introduced by Gabor in 1946, \textcolor{black}{see \cite{Gabor:1946aa}}.
According to \cite[chapter 4.2]{Mallat:2009aa}, for $(x,\xi)\in\R^2$ and a symmetric window function $g\colon\R\to\R$, satisfying $\|g\|_{L^2(\R)}=1$, these atoms are given by
\begin{align*}
	g_{x,\xi}(y)
	:=e^{i\xi y}g(y-x),
	\quad y\in\R.
\end{align*}
The resulting short-time Fourier transform (STFT) of $\psi\in L^2(\R)$ is defined as
\begin{align}\label{eq:stft}
	\operatorname{STFT}(\psi)(x,\xi)
	:=\langle \psi,g_{x,\xi}\rangle_{L^2(\R)}
	=\int_{-\infty}^\infty\ \psi(y) g(y-x)e^{-i\xi y}\ \mathrm{d}y.
\end{align}
It can be understood as the Fourier transform on the real line of the windowed function $\psi\cdot g(\bullet-x)$.
If $g$ is localized in a neighborhood of $x$, then the same applies to the windowed function.
Hence, the spectrum of the STFT is connected to the windowed interval.
In particular, Gabor investigated Gaussian windows with respect to the uncertainty principle, see \cite[chapter 3.1]{Chui:1992aa}.
In many engineering applications, windows are discussed in terms of data weighting and spectral leakage.
Depending on the type of signal, numerous windows have been developed, see e.g. \cite[chapter IV]{Harris:1978aa}.
More recently, in \cite{McKechan:2010aa} a smooth \textcolor{black}{$C^\infty$}-bump window has been suggested for the analysis of gravitational waves.
\textcolor{black}{
It is an essential property of this window, that it is equal to 1 in a closed subinterval of its support (plateau).
Although the Fourier coefficients of such windows may exceed spectral convergence (faster than any fixed polynomial rate), it is their compact support which limits the order to be at most root exponential and the actual convergence rate depends on the growth of the window's derivatives, see e.g. \cite{Tadmor:2007aa}.
For example, in \cite{Boyd:2006aa} a smooth bump is designed such that the order of the windowed Fourier coefficients is root-exponential (at least for the saw wave function), wheres in \cite{Tanner:2006aa} we find a non-compactly supported window, for which we obtain true exponential decay.
We note that Boyd and Tanner focus on an optimal choice of window parameters in order to obtain the best possible approximation results.
}\\

\textcolor{black}{
We investigate the convergence speed of Fourier series windowed by compactly supported bump functions with a plateau.
The properties of these bump windows will allow an effortless transfer of Jackson's classical results on the convergence of the Fourier series for smooth functions.}
\textcolor{black}{
The main new contributions of this paper can be found in Theorem~\ref{fact:error_rep_window} and Theorem~\ref{fact:estimate_lipschitz}, respectively.
In the first one we show that pointwise multiplication (in the time domain) by a window with plateau yields smaller reconstruction errors in the interior of the plateau, compared to those windows without plateau.
We complement this result by a lower error bound for the Hann window, a member of the set of $\cos^\alpha$ functions.
In Theorem~\ref{fact:estimate_lipschitz} we connect the decay rate of windowed Fourier coefficients to a new bound for the variation of windowed functions, which is based on the combination of two main ingredients: the  Leibniz product rule and a bound for intermediate derivatives due to Ore.}

%%%%%%%%%%%%%%%%%%%%
%%%%%%%%%%%%%%%%%%%%
\subsection{Outline}
We start by recalling basic properties of the Fourier series for functions of bounded variation in \S\ref{sec:classical_fourier}.
Afterwards, in \S\ref{sec:windowed_transform} we present the windowed transform, see Proposition~\ref{fact:rep_windowed}, and estimate the reconstruction errors in Theorem~\ref{fact:error_rep_window}.
In \S\ref{sec:bump_windows} we introduce the $C^s$-bumps and transfer the results of chapter~\ref{sec:windowed_transform} to this class.
As a special candidate of $C^1$-bumps, we consider the Tukey window in \S\ref{sub:hann_window}.
Finally, we present numerical experiments in \S\ref{sec:numeric}, that underline our theoretical results and illustrate the benefits of bump windows.
%%%%%%%%%%%%%%%%%%%%%%%%%%%%%%%%%%%%%%%%%%%%%%%%%%%%%%%%%%%%%%%%%
%%%%%%%%%%%%%%%%%%%%%%%%%%%%%%%%%%%%%%%%%%%%%%%%%%%%%%%%%%%%%%%%%
%%%%%%%%%%%%%%%%%%%%%%%%%%%%%%%%%%%%%%%%%%%%%%%%%%%%%%%%%%%%%%%%%
\section{Functions of bounded variation and their Fourier series}\label{sec:classical_fourier}
%%%%%%%%%%%%%%%%%%%%%%%%%%%%%%%%%%%%%%%%%%%
%%%%%%%%%%%%%%%%%%%%%%%%%%%%%%%%%%%%%%%%%%%
\subsection{Functions of bounded variation}
We denote by $\BV_{\text{loc}}$ the set of functions $f\colon\R\to\R$, which are locally of bounded variation, that is of bounded variation on every finite interval.
In particular, we assume that such functions are normalized for any $x$ in the interior of the interval of definition, see \cite[\S0.6]{Butzer:1971aa}, by
\begin{align*}
	f(x)
	=\frac{1}{2}
	\Big(f(x^+)+f(x^-)\Big)
	=\frac{1}{2}
	\left(\lim_{t\to 0+}f(x+t)+\lim_{t\to 0+}f(x-t)\right).
\end{align*}
We recall that a function of bounded variation is bounded, has at most a \textcolor{black}{countable} set of jump discontinuities, and that the pointwise evaluation is well-defined.

%%%%%%%%%%%%%%%%%%%%%%%%%%%%%%%%%%%%%%%%%%%%%%%%%
%%%%%%%%%%%%%%%%%%%%%%%%%%%%%%%%%%%%%%%%%%%%%%%%%
\subsection{The classical Fourier representation}
Any $2\pi$-periodic function $f\in\BV_{\text{loc}}$ has a pointwise converging Fourier series, see \cite[Prop. 4.1.5.]{Butzer:1971aa}.
Let us transfer this representation to an arbitrary interval of length $2\lambda$:
\begin{lemma}(Fourier series of the periodization)\label{fact:rep_fourier}\\
	Suppose that $\psi\in \BV_{\emph{loc}}$ as well as $\lambda>0$ and $t\in\R$.
	Then,
	\begin{align*}
		\psi(x) 
		=\sum_{k\in\Z} c_\psi(k)e^{ik\frac{\pi}{\lambda}x},
		\quad x\in(t-\lambda,t+\lambda),
	\end{align*}
	where the coefficients $c_\psi(k)$ are given by
	\begin{align*}
		c_\psi(k)
		:=\frac{1}{2\lambda}\int_{t-\lambda}^{t+\lambda} \psi(x) e^{-ik\frac{\pi}{\lambda}x}\ \mathrm{d}x	,\quad k\in\Z.
	\end{align*}
\end{lemma}
For the proof of Lemma~\ref{fact:rep_fourier} and our subsequent analysis, we will use a translation, a scaling and a periodization operator.
For the center $t\in\R$ and a scaling factor $a>0$, we introduce:
\begin{align*}
	&\Tra_t\colon\BV_{\text{loc}}\to\BV_{\text{loc}},\ 
	(\Tra_t \psi)(x):= \psi(x+t),\\*[1ex]
	&\Sca_a\colon\BV_{\text{loc}}\to\BV_{\text{loc}},\ 
	(\Sca_a \psi)(x):= \psi(ax).
\end{align*}
For the period half length $\lambda>0$, we set
\begin{align*}	
	&\Per_\lambda\colon\BV_{\text{loc}}\to\BV_{\text{loc}},\\
 	&(\Per_\lambda \psi)(x)
	:=
	\begin{cases}
		\psi(x),
		&\text{if}\ x\in(-\lambda,\lambda),\\[2mm]
		\displaystyle
		\frac{1}{2}\Big(\psi(-\lambda^+)+\psi(\lambda^-)\Big),
		&\text{if}\ x = \lambda.
	\end{cases}
\end{align*}
\begin{proof}
	Consider the $2\pi$-periodic function $f=\Per_\pi\Sca_{\lambda/\pi}\Tra_t\psi$.
	Then, it follows from Lemma~\ref{app:fact_periodization} that $f\in \BV_{\text{loc}}$ and therefore
	\begin{align*}
		f(x) 
		=\sum_{k=-\infty}^\infty \widehat f(k) e^{ikx},\quad x\in\R.
	\end{align*} 
	The Fourier coefficients of $f$ are given by
	\begin{align*}
		\widehat f(k) 
		&= \frac{1}{2\pi}\int_{-\pi}^\pi f(x)e^{-ikx}\ \mathrm{d}x
		=\frac{1}{2\pi} \int_{-\pi}^\pi \left(\Sca_{\lambda/\pi}\Tra_t \psi\right)\!(x) e^{-ikx}\ \mathrm{d}x\\
		&=\frac{1}{2\lambda} \int_{t-\lambda}^{t+\lambda} \psi(x) e^{-ik\frac{\pi}{\lambda}(x-t)}\ \mathrm{d}x.
	\end{align*}
	Consequently, for all $x\in (t-\lambda,t+\lambda)$ we obtain
	\begin{align*}
		\psi(x)
		=\left(\Tra_{-t}\Sca_{\pi/\lambda}f\right)\!(x)
		=\sum_{k\in\Z} c_\psi(k)e^{ik\frac{\pi}{\lambda}x}.
	\end{align*}
\end{proof}
%

%%%%%%%%%%%%%%%%%%%%%%%%%%%%%%%%%%%%%%%%%%%%
%%%%%%%%%%%%%%%%%%%%%%%%%%%%%%%%%%%%%%%%%%%%
\subsection{The classical result of Jackson}
In general, even if $\psi$ is a smooth function, the periodic extension $f=\Per_\pi\Sca_{\lambda/\pi}\Tra_t \psi\in \BV_{\text{loc}}$ has jump discontinuities at $\pm\pi$.
Let $V(f)<\infty$ denote the total variation of $f$.
Then, by \cite[chapter 2.3.6]{Edwards:1982aa},
\begin{align*}
	|k\cdot c_\psi(k)|
	=|k\cdot \widehat f(k)|
	\le \frac{1}{2\pi}V(f),\quad\text{for all $k\in\Z$}.
\end{align*}
Hence, the coefficients are $\mathcal{O}(1/|k|)$.
Moreover, the rate of the coefficients transfers to an estimate for the reconstruction errors.
For an arbitrary function $f\in \BV_{\text{loc}}$ of period $2\pi$ let us introduce the partial Fourier sum
\begin{align*}
	S_nf(x)
	:=\sum_{k=-n}^n \widehat f(k)e^{ikx},
	\quad n\ge 1,\ x\in\R.
\end{align*}
Our analysis relies on the following classical result by Jackson on the convergence of the Fourier sum, see \cite[chapter II.3, Theorem IV]{Jackson:1994aa}:
\begin{proposition}\label{fact:jackson}
	If $f\colon\R\to\R$ is a function of period $2\pi$, which has a $s$th derivative with limited variation, $s\ge 1$, and if $V$ is the total variation of $f^{(s)}$ over a period, then, for $n>0$,
	\begin{align}\label{eq:error_partial_sum}
		|f(x)-S_nf(x)|
		\le \frac{2V}{s\pi n^s},\quad x\in\R.
	\end{align}
\end{proposition}
%
%%%%%%%%%%%%%%%%%%%%%%%%%%%%%%%%
%%%%%%%%%%%%%%%%%%%%%%%%%%%%%%%%
%%%%%%%%%%%%%%%%%%%%%%%%%%%%%%%%
\section{The windowed transform}\label{sec:windowed_transform}
There seems to be no general definition of a window function, but most authors tend to think of a real function $w\ne 0$, vanishing outside a given interval.
In relation to the STFT in \eqref{eq:stft}, additional properties, such as a smooth cut-off or complex values, may be required, see e.g. \cite[\S3]{Grochenig:2001aa} and \cite[\S2]{Kaiser:2011aa}.
Whenever speaking about windows in this paper, we assume the following:
\begin{definition}
	Let $\lambda>0$.
	We say that a function $w\in \BV_{\emph{loc}}$ is a \emph{window function} on the interval $(-\lambda,\lambda)$, if the following properties are satisfied:
	\begin{align}
		\begin{split}
			&(1)\ 0\le w(x)\le 1,\ \text{for}\ x\in (-\lambda,\lambda),\\
			&(2)\ w(x)=0,\ \text{for}\ x\in \R\backslash(-\lambda,\lambda).
		\end{split}
	\end{align}
\end{definition}
In particular, we obtain the rectangular window, if $w(x)=1$ for all $x\in(-\lambda,\lambda)$, and for simplicity we just write $w\equiv 1$ in this case.
For $\psi\in \BV_{\text{loc}}$ and a window $w$ on $(-\lambda,\lambda)$ we introduce the windowed periodization
\begin{align}\label{eq:fw}
	\psi_w:=\Per_\pi\Sca_{\lambda/\pi}\left[w\cdot \Tra_t \psi\right].
\end{align}
Note that $\psi_w$ is $2\pi$-periodic, and by Lemma~\ref{app:fact_periodization} we obtain $\psi_w\in \BV_{\text{loc}}$.

%%%%%%%%%%%%%%%%%%%%%%%%%%%%%%%%%%%%%%%%
%%%%%%%%%%%%%%%%%%%%%%%%%%%%%%%%%%%%%%%%
\subsection{The windowed representation}
According to the classical Fourier series of the periodization presented in Lemma~\ref{fact:rep_fourier}, the windowed series allows an alternative representation with potentially faster convergence.
\begin{proposition}(Windowed Fourier series)\label{fact:rep_windowed}\\
	Let $\psi\in \BV_{\emph{loc}}$ and $\lambda>0$ and $t\in\R$.
	If $w\in \BV_{\emph{loc}}$ is a window on $(-\lambda,\lambda)$, then,
	\begin{align*}
		\psi(x)w(x-t) 
		=\sum_{k\in\Z} c_\psi^w(k)e^{ik\frac{\pi}{\lambda}x},
		\quad x\in(t-\lambda, t+\lambda),
	\end{align*}
	where the coefficients $c_\psi^w(k)$ are given by
	\begin{align*}
		c_\psi^w(k)
		:=\frac{1}{2\lambda}\int_{t-\lambda}^{t+\lambda} \psi(x)w(x-t) e^{-ik\frac{\pi}{\lambda}x}\ \mathrm{d}x,\quad k\in\Z.
	\end{align*}
\end{proposition}
The statement in the last Proposition follows as in Lemma~\ref{fact:rep_fourier}, for the Fourier series of the $2\pi$-periodic windowed shape $\psi_w\in \BV_{\text{loc}}$.\\

Suppose that $\psi_w\in C^s(\R),\ s\ge 1$, and that $\psi_w^{(s)}$ has bounded variation.
Then, as it follows from \cite[chapter II.3, Corollary I]{Jackson:1994aa},
\begin{align}\label{eq:order_windowed}
	|c_\psi^w(k)|
	=|\widehat \psi_w(k)|
	\le \frac{V(\psi_w^{(s)})}{\pi|k|^{s+1}},
	\quad k\ne 0,
\end{align}
and thus the decay rate of the windowed coefficients $c_\psi^w$ improves to $\mathcal{O}\left(1/|k|^{s+1}\right)$.

%%%%%%%%%%%%%%%%%%%%%%%%%%%%%%%%%%%%%%%%%%%%%%%%%%%%%%
%%%%%%%%%%%%%%%%%%%%%%%%%%%%%%%%%%%%%%%%%%%%%%%%%%%%%%
\subsection{An error estimate for the representations}
For $n\ge 1$ and $x\in\R$ let
\begin{align*}
	R_n^w\psi(x)
	:=\sum_{k=-n}^n c_\psi^w(k)e^{ik\frac{\pi}{\lambda}x}
	\quad\text{and}\quad
	R_n\psi(x)
	:=R_n^{w\equiv1}\psi(x)
	=\sum_{k=-n}^n c_\psi(k)e^{ik\frac{\pi}{\lambda}x}.
\end{align*}
Note that $R_n^w\psi=\Tra_{-t}\Sca_{\pi/\lambda}(S_n\psi_w)$.
\textcolor{black}{
We now transfer Jackson's classical result in Proposition~\ref{fact:jackson} to an estimate for the windowed reconstruction errors in terms of the Lipschitz constant of $\psi_w^{(s)}$.
In order not to overload the notation unnecessarily, for the main results in this paper we always assume that $\lambda=\pi$ and $t=0$, that is, both the function $\psi$ and $\psi_w$ are $2\pi$-periodic and centered at the origin.
However, all results could also be formulated for an arbitrary choice of $\lambda>0$ and $t\in\R$ by performing an appropriate scaling and translation.}
\begin{theorem}(\textcolor{black}{Reconstruction, windowed series, $\lambda=\pi$ and $t=0$})\label{fact:error_rep_window}\\
	\textcolor{black}{
	Suppose that $\psi_w\in C^{s+1}(\R),\ s\ge 1$ and let $L_s>0$ denote the Lipschitz constant of $\psi_w^{(s)}$ over $[-\pi,\pi]$.
	Moreover, let $0<\rho<\pi$.
	Then, for $n\ge 1$ the error of the reconstruction $R_n^w\psi$ in the interval $[-\rho,\rho]$ is given by
	\begin{align}\label{eq:error_k_inf}
		\left|\sup_{x\in[-\rho,\rho]}\left|\psi(x)-R_n^w\psi(x)\right|-K_\infty(\psi,w,\rho)\right|
		\le \frac{4L_s}{sn^s},
	\end{align}
	where the non-negative constant $K_\infty(\psi,w,\rho)\ge 0$ is given by 
	\begin{align*}
		K_\infty(\psi,w,\rho)
		=\sup_{x\in[-\rho,\rho]}\Big(|\psi(x)|\cdot\big(1-w(x)\big)\!\Big).
	\end{align*}
	}
\end{theorem}
\begin{proof}
	\textcolor{black}{
	Let $V<\infty$ denote the total variation of $\psi_w^{(s)}$ over a period.
	In particular,
	\begin{align*}
		V
		=\int_{-\pi}^\pi |\psi_w^{(s+1)}(x)|\ \mathrm{d}x
		\le 2\pi L_s.
	\end{align*}
	Hence, for all $x\in\R$ the classical Jackson result in Proposition~\ref{fact:jackson} yields
	\begin{align*}
		A_n(x)
		:=|\psi_w(x)-R_n^w\psi_w(x)|
		=|\psi_w(x)-S_n\psi_w(x)|
		\le \frac{4L_s}{s n^s}.
	\end{align*}
  	Moreover, for all $x\in[-\rho,\rho]$ we have $0\le w(x)\le 1$ and thus, by the reverse triangle inequality, we obtain
	\begin{align}\label{eq:estimates_k_inf}
		\left|\psi(x)-R_n^w\psi(x)\right|
		\begin{Bmatrix}
			\le\\
			\ge
		\end{Bmatrix}
		\left||\psi(x)|\cdot\big(1-w(x)\big)
		\begin{Bmatrix}
			+\\
			-
		\end{Bmatrix}
		A_n(x)\right|,\quad x\in[-\rho,\rho].
	\end{align}
	Taking the supremum proves \eqref{eq:error_k_inf}.
	}
\end{proof}
Note that for $w\equiv 1$ we obtain the convergence of the plain reconstruction $R_n\psi$, where $\color{black}{K_\infty(\psi,w,\rho)}=0$.
Theorem~\ref{fact:error_rep_window} allows a calculation of the $L^2$-error:
\begin{corollary}\label{fact:error2_rep_window}
	\textcolor{black}{
	The $L^2$-error of the reconstruction is given by
	\begin{align}\label{eq:error_k2}
		\left|\left\|\psi-R_n^w\psi\right\|^2_{L^2([-\rho,\rho])}-K_2(\psi,w,\rho)\right|
		&\le \frac{16\rho L_s}{sn^s}K_\infty(\psi,w,\rho)+\frac{32\rho L_s^2}{s^2n^{2s}},
	\end{align}
	where the non-negative constant $K_2(\psi,w,\rho)\ge 0$ is given by 
	\begin{align*}
		K_2(\psi,w,\rho)
		=\int_{-\rho}^\rho|\psi(x)|^2\big(1-w(x)\big)^2\ \mathrm{d}x.
	\end{align*}
	In particular, $K_2(\psi,w,\rho)=0$, if and only if $K_\infty(\psi,w,\rho)=0$.
	}
\end{corollary}
\begin{proof}
	\textcolor{black}{For $p\in\{1,2\}$ we introduce $N_{p,n,\rho}:=\left\|\psi_w-R^w_n\psi\right\|_{L^p([-\rho,\rho])}$.
	Then, it follows from \eqref{eq:estimates_k_inf} that for all $x\in[-\rho,\rho]$:
	\begin{align*}
		\left|\psi(x)-R_n^w\psi(x)\right|^2
		\begin{Bmatrix}
			\le\\
			\ge
		\end{Bmatrix}
		\left||\psi(x)|\cdot\big(1-w(x)\big)
		\begin{Bmatrix}
			+\\
			-
		\end{Bmatrix}
		A_n(x)\right|^2,
	\end{align*}
	and therefore, integration yields
	\begin{align*}
		\left\|\psi-R_n^w\psi\right\|^2_{L^2([-\rho,\rho])}
		\begin{Bmatrix}
			\le\\
			\ge
		\end{Bmatrix}
		K_2(\psi,w,\rho)
		\begin{Bmatrix}
			+\\
			-
		\end{Bmatrix}
		2K_\infty(\psi,w,\rho)N_{1,n,\rho}+N_{2,n,\rho}^2.
	\end{align*}
	Consequently, \eqref{eq:error_k2} follows from
	\begin{align*}
		N_{1,n,\rho}
		\le \sqrt{2\rho}\cdot N_{2,n,\rho}
		\le 2\rho\cdot\left(\sup_{x\in[-\rho,\rho]}A_n(x)\right)
		\le \frac{8\rho L_s}{sn^s}.
	\end{align*}
	}	
\end{proof}
In addition to the assumptions in Theorem~\ref{fact:error_rep_window}, let us assume that $w(x)=1$ for all $x\in[-\rho,\rho]$.
Then, it follows that $\color{black}{K_\infty(\psi,w,\rho)}=0$ and therefore $\color{black}{K_2(\psi,w,\rho)}=0$.
Hence, the reconstruction errors converge to 0 as $n\to\infty$.
This motivates the investigation of bump windows.
%%%%%%%%%%%%%%%%%%%%%%
%%%%%%%%%%%%%%%%%%%%%%
%%%%%%%%%%%%%%%%%%%%%%
\section{Bump windows}\label{sec:bump_windows}
\textcolor{black}{We now introduce $C^s$-bump windows by singling out two additional properties:}
On the one hand, bump windows fall off smoothly at the boundary of their support, on the other hand, to receive a faithful windowed shape of the original function, bump windows have to equal 1 in a closed subinterval of their support.
The plots in Figure~\ref{fig:windows} show the typical shape and the action of a bump.

%%%%%%%%%%%%%%%%%%%%%%%%%%%%%%%%%%%
%figure
\begin{figure}
	\includegraphics{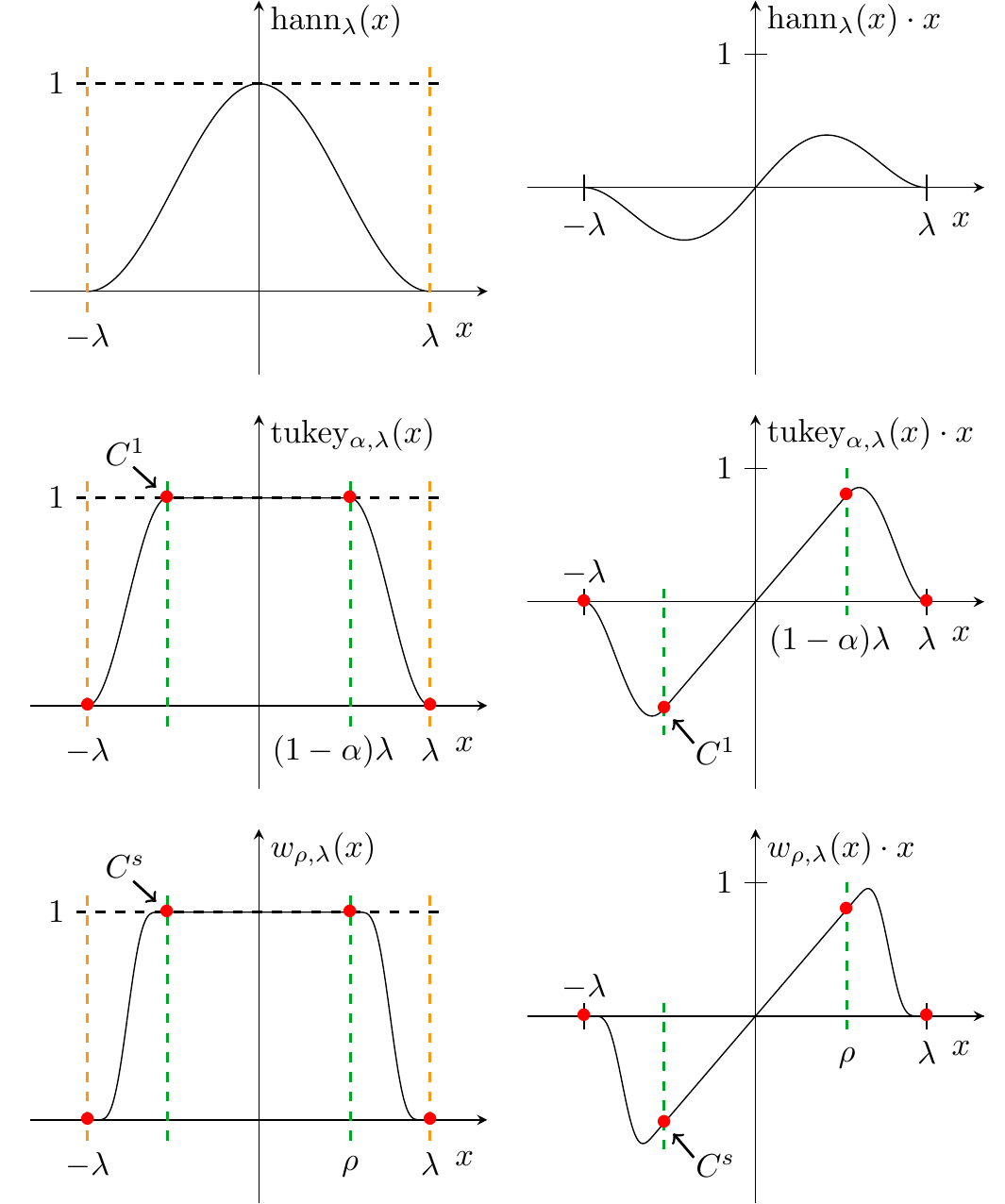}
	\caption{Different bump windows (left) and their action on the function $\psi(x)=x$ (right).
		The Hann window, see Definition~\ref{def:hann_window}, can be viewed as a \textcolor{black}{degenerate} $C^1$-bump.
		For $0<\alpha<1$, the Tukey window, see Definition~\ref{def:tukey_window}, is a \textcolor{black}{non-degenerate} $C^1$-bump.
		Generally, the $C^s$-bump $w_{\rho,\lambda}$ in Definition~\ref{def:bump} (bottom) is $s$-times, but not ($s+1$)-times continuously differentiable.}
	\label{fig:windows}	
\end{figure}
%%%%%%%%%%%%%%%%%%%%%%%%%%%%%%%%%%%

%
\begin{definition}\label{def:bump}
	Let $\lambda>0$ and $0\le\rho<\lambda$.
	For some $s\ge 1$ we say that the function $w_{\rho,\lambda}\in C_c^s(\R)$ is a $C^s$-\emph{bump}, if the following properties are satisfied:
	\begin{equation*}
		\begin{split}
			&(1)\ 0\le {\color{black}{w_{\rho,\lambda}(x)}}\le 1,\ \text{for}\ x\in (-\lambda,\lambda),\\
			&(2)\ {\color{black}{w_{\rho,\lambda}(x)}}=0,\ \text{for}\ x\in \R\backslash(-\lambda,\lambda),\\
			&(3)\ w_{\rho,\lambda}(x)=1,\ \text{for}\ x\in[-\rho,\rho].
		\end{split}
	\end{equation*}
	If $\rho=0$, we say that the bump is \emph{\textcolor{black}{degenerate}}.
	Moreover, whenever $w_{\rho,\lambda}\in C_c^\infty(\R)$, we say that $w_{\rho,\lambda}$ is a \emph{smooth bump}.
\end{definition}
\textcolor{black}{We note that smooth bump windows have previously been used for data analysis of gravitational waves, see \cite[Equation (3.35)]{Damour:2000aa} and \cite[\S2, Equation (7)]{McKechan:2010aa}.}
Moreover, bump functions occur when working with partitions of unity, e.g. in the theory of manifolds, see \cite[Lemma 2.22]{Lee:2013aa} and \cite[\S13.1]{Tu:2011aa}, and further, with a view to numerical applications, for so-called partition of unity methods, which are used for solving partial differential equations, see \cite[\S 4.1.2]{Griebel:2000aa}.
\textcolor{black}{An example for a smooth bump is given by the even function}
\begin{align}\label{eq:def_bump_smooth}
	w_{\rho,\lambda}(x)
	=
	\begin{cases}
		1,\ 
		& \text{if $0 \le|x|\le \rho$},\\
		\displaystyle\frac{1}{\exp\left(\frac{1}{\lambda-|x|}+\frac{1}{\rho-|x|}\right)+1},\ 
		& \text{if $\rho<|x|<\lambda$},\\
		0,\ 
		& \text{if $|x|\ge\lambda$}.
	\end{cases}
\end{align}
As we see in the right plots of Figure~\ref{fig:windows}, the product of a \textcolor{black}{non-degenerate} bump $w_{\rho,\lambda}$ and $\psi$ produces a (smooth) windowed shape, matching with $\psi$ in $[-\rho,\rho]$ and tending to 0 at the boundaries of $(-\lambda,\lambda)$.
In particular, we obtain excellent reconstructions using the smooth bump given by \eqref{eq:def_bump_smooth} in our numerical experiments.\\

\textcolor{black}{Although in this paper we only consider compactly supported windows, we note that also other types have been studied extensively in the past.
By abandoning the compact support, smooth windows can potentially be used for the pointwise reconstruction of exponential accuracy, wheres windows with compact support can at most obtain root exponential accuracy, see also \S\ref{sub:rep_bump_windows}.
An example for windows not having compact support is given by the class of exponential functions, which are of the form $\exp(-cx^{2m})$, where $c>0$ is a positive real constant and $m\ge 1$ is a positive integer.
We note that the choice of $c$ and $m$ control the decay of the window and from a numerical point of view, due to machine tolerance, it can be argued that a computer treats it as being compact, see \cite[\S2]{Tanner:2006aa}.
For other examples of non-compactly supported windows we refer to the work of Boyd in \cite{Boyd:1996aa} and subsequent papers, who pioneered the concept of adaptive filters.}

%%%%%%%%%%%%%%%%%%%%%%%%%%%%%%%%%%%%%%%%%%%
%%%%%%%%%%%%%%%%%%%%%%%%%%%%%%%%%%%%%%%%%%%
\subsection{The Hann- and the Tukey window}\label{sub:hann_window}
The class of bump windows includes the famous Hann window, which can be defined as follows, see \cite[\S4.2.2]{Mallat:2009aa}:
\begin{definition}\label{def:hann_window}
	Let $\lambda>0$.
	For all $x\in\R$ the \emph{Hann window} is given by
	\begin{align*}
		\operatorname{hann}_\lambda(x)
		=\cos^2\left(\frac{\pi}{2\lambda}x\right)\cdot\mathbf{1}_{[0,\lambda]}(|x|)
		=\frac{1}{2}\left[1+\cos\left(\frac{\pi}{\lambda}x\right)\right]\cdot\mathbf{1}_{[0,\lambda]}(|x|).
	\end{align*}
\end{definition}
In the sense of Definition~\ref{def:bump} the Hann window is a \textcolor{black}{degenerate} $C^1$-bump.
In particular, for $0<\rho'<\lambda$ it follows from Theorem~\ref{fact:error_rep_window} and Corollary~\ref{fact:error2_rep_window}, that the reconstruction errors for a function $\psi\ne 0$ \textcolor{black}{on} the interval $[t-\rho',t+\rho']$ are bounded from below by positive constants $\color{black}{K_\infty(\psi,w,\rho'),K_2(\psi,w,\rho')}>0$.
This fact can also be observed in our numerical experiments, see \S\ref{sub:saw_wave} and \S\ref{sub:parabola}.
\textcolor{black}{We note that the Hann window is a famous representative of windows specially used in signal processing.}\\

As it turns out, the Hann window arises as a special candidate of a more general class, the Tukey windows, \textcolor{black}{see \cite{Tukey:1967aa}}, often called cosine-tapered windows.
These windows can be imagined as a cosine lobe convolved with a rectangular window:
\begin{definition}\label{def:tukey_window}
	The \emph{Tukey window} with parameter $\alpha\in (0,1]$ is given by
	\begin{align*}
		\operatorname{tukey}_{\alpha,\lambda}(x)
		:=\mathbf{1}_{[0,(1-\alpha)\lambda)}(|x|)+\frac{1}{2}\left[1-\cos\left(\frac{\pi|x|}{\alpha\lambda}-\frac{\pi}{\alpha}\right)\right]\cdot\mathbf{1}_{[(1-\alpha)\lambda,\lambda]}(|x|).
	\end{align*}
\end{definition}
The Tukey window is a $C^1$-bump $w_{\rho,\lambda}$ with $\rho=(1-\alpha)\lambda$.
In particular,
\begin{align*}
	\operatorname{tukey}_{1,\lambda}
	=\operatorname{hann}_\lambda
	=w_{0,\lambda},
\end{align*}
and for $0<\alpha<1$ the Tukey window is not \textcolor{black}{degenerate}.
We note that the sum of phase-shifted Hann windows creates a Tukey window:
\begin{lemma}
	Let $\tau>0$ and $m\ge 0$.
	Then, for $\alpha=1/(m+1)$ and $\lambda=(m+1)\tau$,
	\begin{align*}
		\sum_{k=-m}^m\operatorname{hann}_\tau(\bullet-k\tau)
		=\operatorname{tukey}_{\alpha,\lambda}.
	\end{align*}
\end{lemma}
\begin{proof}
	For all $x\in\R$ we introduce the function
	\begin{align*}
		H_{\tau,m}(x)
		:=\sum_{k=-m}^m\operatorname{hann}_\tau(x-k\tau).
	\end{align*}
	Obviously, $H_{\tau,m}$ is an even function.
	Thus, for all $x\in\R$ we obtain
	\begin{align*}
		H_{\tau,m}(x)
		&=\sum_{k=0}^m\operatorname{hann}_\tau(|x|-k\tau)
		=\frac{1}{2}\sum_{k=0}^m\left[1+\cos\left(\frac{\pi}{\tau}(|x|-k\tau)\right)\right]\cdot\mathbf{1}_{[-\tau,\tau]}(|x|-k\tau)\\
		&=\frac{1}{2}\sum_{k=0}^m\left[1+\cos\left(\frac{\pi}{\tau}(|x|-k\tau)\right)\right]\cdot\Big(\mathbf{1}_{[(k-1)\tau,k\tau)}(|x|)+\mathbf{1}_{[k\tau,(k+1)\tau)}(|x|)\Big)\\
		&=\frac{1}{2}\sum_{k=0}^m\left[1+\cos\left(\frac{\pi}{\tau}(|x|-k\tau)\right)\right]\cdot\mathbf{1}_{[(k-1)\tau,k\tau)}(|x|)\\
		&\qquad+\frac{1}{2}\sum_{k=1}^{m+1}\left[1+\cos\left(\frac{\pi}{\tau}(|x|-(k-1)\tau)\right)\right]\cdot\mathbf{1}_{[(k-1)\tau,k\tau)}(|x|)\\
		&=\mathbf{1}_{[0,m\tau)}(|x|)+\frac{1}{2}\left[1-\cos\left(\frac{\pi}{\tau}(|x|-(m+1)\tau)\right)\right]\cdot\mathbf{1}_{[m\tau,(m+1)\tau)}(|x|)\\
		&\qquad+\frac{1}{2}\sum_{k=1}^m\left(\cos\left(\frac{\pi}{\tau}(|x|-k\tau)\right)+\cos\left(\frac{\pi}{\tau}(|x|-k\tau)+\pi\right)\right)\cdot\mathbf{1}_{[(k-1)\tau,k\tau)}(|x|)\\
		&=\mathbf{1}_{[0,m\tau)}(|x|)+\frac{1}{2}\left[1-\cos\left(\frac{\pi}{\tau}(|x|-(m+1)\tau)\right)\right]\cdot\mathbf{1}_{[m\tau,(m+1)\tau)}(|x|).
	\end{align*}
\end{proof}

%%%%%%%%%%%%%%%%%%%%%%%%%%%%%%%%%%%%%%%%%%%%%%%%
%%%%%%%%%%%%%%%%%%%%%%%%%%%%%%%%%%%%%%%%%%%%%%%%
\subsection{The representation for bump windows}\label{sub:rep_bump_windows}
The windowed Fourier series in Proposition~\ref{fact:rep_windowed} applies to bump functions and yields the following representation in the restricted interval $[t-\rho,t+\rho]$:
\begin{corollary} (Fourier series windowed by a bump function)\label{fact:rep_bump_window}\\
	Suppose that $\psi\in C^{s+1}(\R),\ s\ge 1$, as well as $\lambda>0$ and $0\le\rho<\lambda$ and $t\in\R$.
	If $w_{\rho,\lambda}\in C_c^{s+1}(\R)$ is a $C^{s+1}$-bump on $(-\lambda,\lambda)$, satisfying the three conditions in Definition~\ref{def:bump}, then,
	\begin{align*}
		\psi(x)
		=\sum_{k\in\Z} c_\psi^w(k)e^{ik\frac{\pi}{\lambda}x},
		\quad x\in[t-\rho, t+\rho],
	\end{align*}
	where the coefficients $c_\psi^w(k)$ are given by
	\begin{align*}
		c_\psi^w(k)
		=\frac{1}{2\lambda}\int_{t-\lambda}^{t+\lambda} \psi(x)w_{\rho,\lambda}(x-t) e^{-ik\frac{\pi}{\lambda}x}\ \mathrm{d}x,\quad k\in\Z.
	\end{align*}
	In particular, if $L_s>0$ denotes the Lipschitz constant of $\psi_w^{(s)}$ over $[-\pi,\pi]$, then,
	\begin{align}\label{eq:estimate_coef_bump}
		|c_\psi^w(k)|
		\le\frac{V(\psi_w^{(s)})}{\pi|k|^{s+1}}
		\le \frac{2L_s}{|k|^{s+1}},
		\quad k\ne 0.
	\end{align}
\end{corollary}
We note that for $w=\operatorname{hann}_\lambda$ the representation in Corollary~\ref{fact:rep_bump_window} shrinks to a pointwise representation at $x=t$.
Furthermore, the bound in \eqref{eq:estimate_coef_bump} depends on the choice of the bump $w_{\rho,\lambda}$, and for $\rho\approx\lambda$ the windowed transform does not lead to an improvement of the decay for low frequencies $k$, because in this case the action of the bump is comparable to a truncation of $\psi$, such that the Lipschitz constant $L_s$ dominates.
We will illustrate this fact with numerical experiments in \S\ref{sub:parabola}.

\textcolor{black}{
Moreover, we note that for a smooth bump $w_{\rho,\lambda}\in C_c^\infty(\R)$ the coefficients $c_\psi^w(k)$ do not decay exponentially fast, since the window is compactly supported and thus not analytic, see \cite{Tadmor:1986aa}.
Nevertheless, the coefficients of a smooth bump have an exponential rate of fractional order and the actual rate can be classified by analyzing their so-called Gevrey regularity, see \cite[Equation (2.4)]{Tadmor:2007aa}.
}

%%%%%%%%%%%%%%%%%%%%%%%%%%%%%%%%%%%%%%%%%%%%%%%
%%%%%%%%%%%%%%%%%%%%%%%%%%%%%%%%%%%%%%%%%%%%%%%
\subsection{A bound for the Lipschitz constant}
We now investigate the Lipschitz constant $L_s$ in Corollary~\ref{fact:rep_bump_window}.
Using the work of Ore in \cite{Ore:1938aa}, we crucially use an estimate on the higher order derivatives of the product of two functions, which is developed in \S\ref{subsec:product}.\\

For a function $f\colon\R\to\R$, that is $(s+1)$-times differentiable, $s\ge 1$, with a $(s+1)$th derivative bounded on a finite interval $(a,b)$, let us introduce the non-negative constant
\begin{align}\label{eq:Cf}
	C_{s,f} 
	=\sup_{x\in(a,b)}|f(x)|+\frac{(b-a)^{s+1}}{(s+1)!}\sup_{x\in(a,b)}|f^{(s+1)}(x)|
	\ge 0.
\end{align}
\begin{theorem}(\textcolor{black}{Bound for the Lipschitz constant, $\lambda=\pi$ and $t=0$})\label{fact:estimate_lipschitz}\\
	\textcolor{black}{
	Let $0\le\rho<\pi$ and suppose that $\psi\in C^{s+1}(\R)$ and $w_{\rho,\pi}\in C_c^{s+1}(\R)$ for some $s\ge 1$.
	Assume the existence of two non-negative constants $M_\psi, M_{\psi_{s+1}}\ge 0$, such that
	\begin{align*}
		|\psi(x)|\le M_\psi
		\quad\text{and}\quad 
		|\psi^{(s+1)}(x)|\le M_{\psi_{s+1}}
		\quad\text{for all $x\in (-\pi,\pi)$.}
	\end{align*}
	Then, the Lipschitz constant $L_s$ in Corollary~\ref{fact:rep_bump_window} is bounded by 
	\begin{align*}
		L_s
		\le M_{\psi_{s+1}}+M_\psi\|w^{(s+1)}_{\rho,\pi}\|_\infty+\frac{C_{s,\psi} C_{s,w}}{(2\pi)^{s+1}}\cdot K_s,
	\end{align*}
	where the non-negative constants $C_{s,\psi},C_{s,w}\ge 0$ are given by
	\begin{align*}
		C_{s,\psi}
		=M_\psi+\frac{(2\pi)^{s+1}}{(s+1)!}M_{\psi_{s+1}}
		\quad\text{and}\quad
		C_{s,w}
		=1+\frac{(2\pi)^{s+1}}{(s+1)!}\|w^{(s+1)}_{\rho,\pi}\|_\infty,
	\end{align*}
	and the constant $K_s>0$ is given by
	\begin{align}\label{eq:formula_ks}
		K_s
		=\frac{2^{2s+1}\cdot s^2\cdot (3s)!}{ (2s+1)!}.
	\end{align}
	}
\end{theorem}
\begin{proof}
	\textcolor{black}{
	According to Proposition~\ref{fact:prod} in the next section, we use the bound for the $(s+1)$th derivative of the product $fg$ for $f=w_{\rho,\pi}$ and $g=\psi$.
	This results in
	\begin{align*}
		L_s
		=\sup_{x\in(-\pi,\pi)}\left|\frac{\mathrm{d}^{s+1}}{\mathrm{d}x^{s+1}}\Big(w_{\rho,\pi}(x)\psi(x)\Big)\right|
		\le M_{\psi_{s+1}}+M_\psi\|w^{(s+1)}_{\rho,\pi}\|_\infty+\frac{C_{s,\psi}C_{s,w}}{(2\pi)^{s+1}}\cdot K_s.
	\end{align*}
	Moreover, for the formula of the constant $K_s$ in \eqref{eq:formula_ks} we refer to Lemma~\ref{fact:combinatorial_constant}.
	}
\end{proof}
\begin{remark}
	Stirling's formula yields the following approximation of $K_s$: 
	\begin{align*}
		K_s 
		=\frac{2s}{2s+1}\frac{2^{2s}\cdot s\cdot(3s)!}{(2s)!}
		\sim \frac{4^s\cdot s\cdot\sqrt{6\pi s}\cdot(3s)^{3s}e^{-3s}}{\sqrt{4\pi s}\cdot (2s)^{2s}e^{-2s}}
		=s\cdot\sqrt{\frac{3}{2}}\left(\frac{27s}{e}\right)^s.
	\end{align*}
	The sign $\sim$ means that the ratio of the quantities tends to 1 as $s\to\infty$.
\end{remark}
\textcolor{black}{
In \cite[Lemma 3.2]{Gottlieb:1985aa}, Gottlieb and Tadmor present a bound for the largest maximum norm of a windowed Dirichlet kernel (regularization kernel) and its first $s$ derivatives.
This bound is used to derive an error estimate for the reconstruction of a function by a discretization of the convolution integral with an appropriate trapezoidal sum, cf. \cite[Proposition 4.1]{Gottlieb:1985aa}.}
\textcolor{black}{
Instead of working with the largest maximum norm of the first $s$ derivatives, we are now presenting a new bound for the $(s+1)$th derivative of a product of two functions.
We therefore combine the Leibniz product rule with individual bounds for intermediate derivatives, and to the best of our knowledge, this is the first time that an explicit bound has been revealed this way.
}

%%%%%%%%%%%%%%%%%%%%%%%%%%%%%%%%%%%%%%%%%%%%%%%%%%%%%%%%%%%%%
%%%%%%%%%%%%%%%%%%%%%%%%%%%%%%%%%%%%%%%%%%%%%%%%%%%%%%%%%%%%%
\subsection{Estimating higher order derivatives of a product}\label{subsec:product}
If $f$ is $(s+1)$-times differentiable, and if its $(s+1)$th derivative is bounded on a finite interval $(a,b)$, then, it follows from \cite[Theorem 2]{Ore:1938aa} that all intermediate derivatives are bounded.
In particular, for all $i=1,...,s$ and all $x\in(a,b)$,
\begin{align}\label{eq:fi}
	|f^{(i)}(x)|
	\le K(i,s)\cdot\frac{C_{s,f}}{(b-a)^i},
\end{align}
where the combinatorial constant $K(i,s)>0$ is defined according to
\begin{align}\label{eq:def_k}
	K(i,s)
	=\frac{2^i\cdot s^2(s^2-1^2)\cdots(s^2-(i-1)^2)}{1\cdot3\cdot5\cdots(2i-1)},\quad i\in \{1,...,s\}.
\end{align}
We now use the general Leibniz rule to lift this result to an explicit bound for the $(s+1)$th derivative of the product of two functions.
\begin{proposition}\label{fact:prod}
	Let $s\ge 1$ and $f,g\colon \R \to\R$, both $(s+1)$-times differentiable in a finite interval $(a,b)$. 
	Assume the existence of four non-negative constants
	\begin{align*}
		M_f, M_g, M_{f_{s+1}}, M_{g_{s+1}}
		\ge 0,
	\end{align*}
	such that for all $x\in(a,b)$:
	\begin{align*}
		|f(x)|
		\le M_f,\
		|g(x)|
		\le M_g
		\quad and\quad
		|f^{(s+1)}(x)|
		\le M_{f_{s+1}},\
		|g^{(s+1)}(x)|
		\le M_{g_{s+1}}.
	\end{align*}
	Then, for all $x\in(a,b)$ we have
	\begin{align*}
		|(fg)^{(s+1)}(x)|
		\le M_f M_{g_{s+1}}+M_{f_{s+1}}M_g+\frac{C_{s,f}C_{s,g}}{(b-a)^{s+1}}\cdot K_s,
	\end{align*}
	where the constants $C_{s,f},C_{s,g}\ge 0$ are defined according to \eqref{eq:Cf} and the constant $K_s>0$, which only depends on $s$, is given by
	\begin{align}\label{eq:Kn}
		K_s
		=\sum_{k=1}^s\binom{s+1}{k} K(s+1-k,s)\cdot K(k,s).
	\end{align} 
\end{proposition}
\begin{proof}
	By the general Leibniz rule the $(s+1)$th derivative of $fg$ is given by
	\begin{align*}
		(fg)^{(s+1)}
		=\sum_{k=0}^{s+1}\binom{s+1}{k}f^{(s+1-k)}\cdot g^{(k)}.
	\end{align*}
	We therefore obtain the following estimate for all $x\in(a,b)$:
	\begin{align*}
		|(fg)^{(s+1)}(x)|
		&\le \sum_{k=0}^{s+1}\binom{s+1}{k}|f^{(s+1-k)}(x)|\cdot |g^{(k)}(x)|\\
		&\le M_f M_{g_{s+1}}+M_{f_{s+1}}M_g+\sum_{k=1}^{s}\binom{s+1}{k}|f^{(s+1-k)}(x)|\cdot |g^{(k)}(x)|.
	\end{align*}
	Using \eqref{eq:fi} for $1\le k\le s$, we conclude that
	\begin{align*}
		|f^{(s+1-k)}(x)|
		&\le K(s+1-k,s)\cdot\frac{C_{s,f}}{(b-a)^{s+1-k}},\\
		|g^{(k)}(x)|
		&\le K(k,s)\cdot\frac{C_{s,g}}{(b-a)^k},
	\end{align*}
	and thus 
	\begin{align*}
		|(fg)^{(s+1)}(x)|
		\le M_f M_{g_{s+1}}+M_{f_{s+1}}M_g+\frac{C_{s,f}C_{s,g}}{(b-a)^{s+1}}\cdot K_s.
	\end{align*}
\end{proof}
\begin{remark}
	\textcolor{black}{The bound
	\begin{align*}
		|f^{(i)}(x)|
		\le K(i,s)\cdot\frac{M_f}{(b-a)^i},\quad x\in(a,b),
	\end{align*}
	for a polynomial $f$ of degree $s$ is due to W. Markoff (1916) and it is known that the equality sign is attained for the Chebyshev polynomials, see \cite{Markoff:1916ab}.}
\end{remark}
%

%%%%%%%%%%%%%%%%%%%%%%%%%%%%%%%%%%%%%%%
%%%%%%%%%%%%%%%%%%%%%%%%%%%%%%%%%%%%%%%
\subsection{The combinatorial constant}\label{subsec:constant}
Next, we will investigate the combinatorial constant $K_s$ and derive formula \eqref{eq:formula_ks} presented in Theorem~\ref{fact:estimate_lipschitz}.
\begin{lemma}\label{fact:combinatorial_constant}
	Let $s\ge 1$.
	The combinatorial constant $K_s>0$ in \eqref{eq:Kn} satisfies
	\begin{align}\label{eq:rep_ks}
		K_s
		=\frac{2^{2s+1}\cdot s^2\cdot (3s)!}{ (2s+1)!}.
	\end{align}
\end{lemma}
\begin{proof}
	We start by rewriting the constant $K(i,s)$ that has been defined in \eqref{eq:def_k}.
	Let $i\in\{1,...,s\}$.
	For the numerator we obtain
	\begin{align*}
		2^i \cdot s^2\cdot(s^2-1^2)\cdots(s^2-(i-1)^2)
		=2^i \cdot \frac{s}{s+i}\cdot \frac{(s+i)!}{(s-i)!}.
	\end{align*}
	For the denominator we have
	\begin{align*}
		1\cdot3\cdot5\cdots(2i-1)
		=\frac{(2i-1)!}{2^{i-1}\cdot (i-1)!}
		=\frac{(2i)!}{2^i\cdot i!}.
	\end{align*}
	Hence, we can rewrite $K(i,s)$ as
	\begin{align*}
		K(i,s)
		=2^i \cdot \frac{s}{s+i}\cdot \frac{(s+i)!}{(s-i)!}\frac{2^i\cdot i!}{(2i)!}
		=\frac{s}{s+i}\cdot 2^{2i}\cdot i!\cdot\binom{s+i}{2i},
	\end{align*}
	and the summands that define the number $K_s$ in \eqref{eq:Kn} can be expressed as
	\begin{align*}
		&\binom{s+1}{k} K(s+1-k,s)\cdot K(k,s)\\
		=\ &2^{2s}\cdot(s+1)!\cdot (2s)^2\cdot\frac{(s+k-1)!\cdot(2s-k)!}{(2k)!\cdot(s-k)!\cdot(2s-2k+2)!\cdot(k-1)!}.
	\end{align*}
	Therefore we conclude that
	\begin{align}\label{eq:rep_ks_sum}
		K_s 
		&=2^{2s}\cdot(s+1)!\cdot\sum_{k=0}^{s-1} \frac{(s+k)!\cdot(2s-k-1)!\cdot(2s)^2}{(2k+2)!\cdot(s-k-1)!\cdot(2s-2k)!\cdot k!}\\
		&=2^{2s}\cdot(s+1)!\cdot\sum_{k=0}^{s-1}\left[\binom{2s-k}{k} \frac{2s}{2s-k}\cdot\binom{s+k+1}{s-k-1} \frac{2s}{s+k+1}\right].\notag
	\end{align}
	Finally, let us introduce
	\begin{align*}
		\kappa_j(2s) 
		=\binom{2s-j}{j} \frac{2s}{2s-j}\quad \text{for}\quad j=0,1,...,s-1.
	\end{align*}
	Recognizing our constant $K_s$ as a Vandermonde-type convolution and using the representation in \cite[Equation (4)]{Gould:1956aa} we write
	\begin{align*}
		K_s 
		&= 2^{2s}\cdot(s+1)!\cdot\sum_{k=0}^{s-1} \kappa_{k}(2s)\cdot \kappa_{s-k-1}(2s)
		=2^{2s}\cdot(s+1)!\cdot \kappa_{s-1}(4s)\\
		&= 2^{2s}\cdot(s+1)!\cdot \binom{3s+1}{s-1}\frac{4s}{3s+1}
		= \frac{2^{2s+1}\cdot s^2\cdot (3s)!}{ (2s+1)!}.
	\end{align*}
\end{proof}
In Appendix~\ref{sec:upper_bound_appendix} we derive an upper bound for $K_s$ based on binomial coefficients.
%%%%%%%%%%%%%%%%%%%%%%%%%%%
%%%%%%%%%%%%%%%%%%%%%%%%%%%
%%%%%%%%%%%%%%%%%%%%%%%%%%%
\section{Numerical results}\label{sec:numeric}
According to our results in Theorem~\ref{fact:error_rep_window} and Corollary~\ref{fact:rep_bump_window} we present numerical experiments for three different functions.
We investigate reconstructions with the smooth bump $w_{\rho,\lambda}$ given by \eqref{eq:def_bump_smooth}, compared to those with the Hann window in Definition~\ref{def:hann_window} and the Tukey window in Definition~\ref{def:tukey_window}.
Besides the reconstructions we also present the decay of the coefficients and the reconstruction errors.\\
In \S\ref{sub:saw_wave} we start with the saw wave function to demonstrate the superiority of the windowed transform with a smooth bump for a function having a high jump discontinuity.
Afterwards, the experiments in \S\ref{sub:parabola} deal with a parabola function.
The symmetric periodic extension has no discontinuities, and therefore the parabola is a good candidate to illustrate the limitations of bump windows.
Last, in \S\ref{sub:hermite} we work with a rapidly decreasing function.
\textcolor{black}{As we will see in this example, for low frequencies all coefficients (plain, tukey, bump) have a rapid \textcolor{black}{initial} decrease, implying excellent reconstructions.}
\begin{remark}
	In the following experiments, the dependency of the windows on the parameters $\lambda, \rho$ and $\alpha$ are always assumed implicitly and therefore we write
	\begin{align*}
		\operatorname{hann}
		=\operatorname{hann}_\lambda,\quad
		\operatorname{tukey}
		=\operatorname{tukey}_{\alpha,\lambda},\quad
		\operatorname{bump}
		=w_{\rho,\lambda}.
	\end{align*}
\end{remark}
For the numerical computation of the (windowed) coefficients we used the fast Fourier transform (FFT), see Appendix~\ref{sec:fft_appendix}.

%%%%%%%%%%%%%%%%%%%%%%%%%%%%%%
%%%%%%%%%%%%%%%%%%%%%%%%%%%%%%
\subsection{Saw wave function}\label{sub:saw_wave}
%%%%%%%%%%%%%%%%%%%%%%%%%%%%%%%%%%%
%figure
\begin{figure}
	\includegraphics{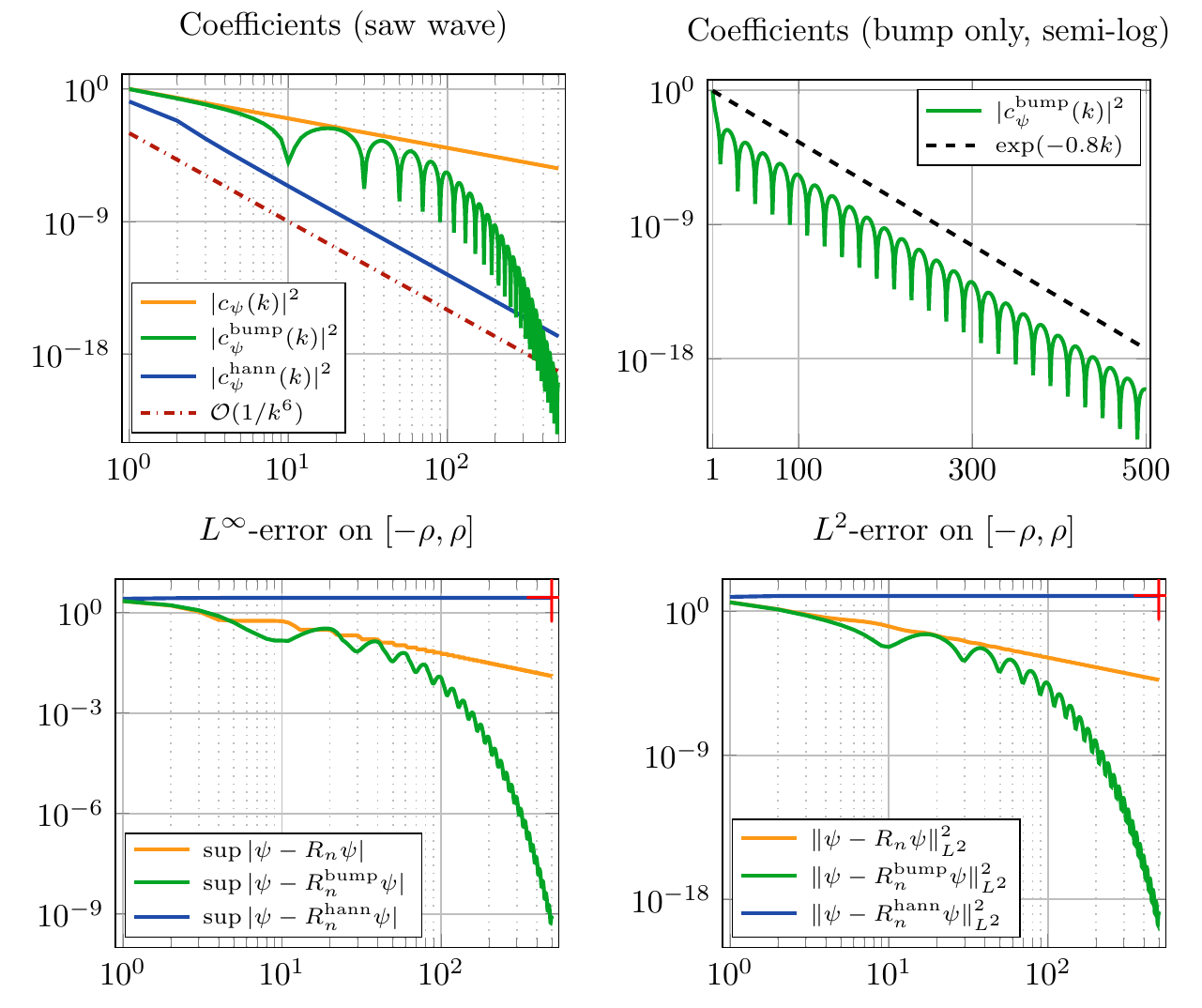}
	\caption{Decay of the coefficients (above) and reconstruction errors (below) for the saw wave. 
		The plain coefficients (orange) have order $\mathcal{O}(1/|k|)$, while the coefficients for the bump (green) show exponential decay (upper right side).
		For the Hann window the errors converge to constant values larger than 1 (red crosses).}
	\label{fig:saw}
\end{figure}
%%%%%%%%%%%%%%%%%%%%%%%%%%%%%%%%%%%
%figure
\begin{figure}
	\includegraphics{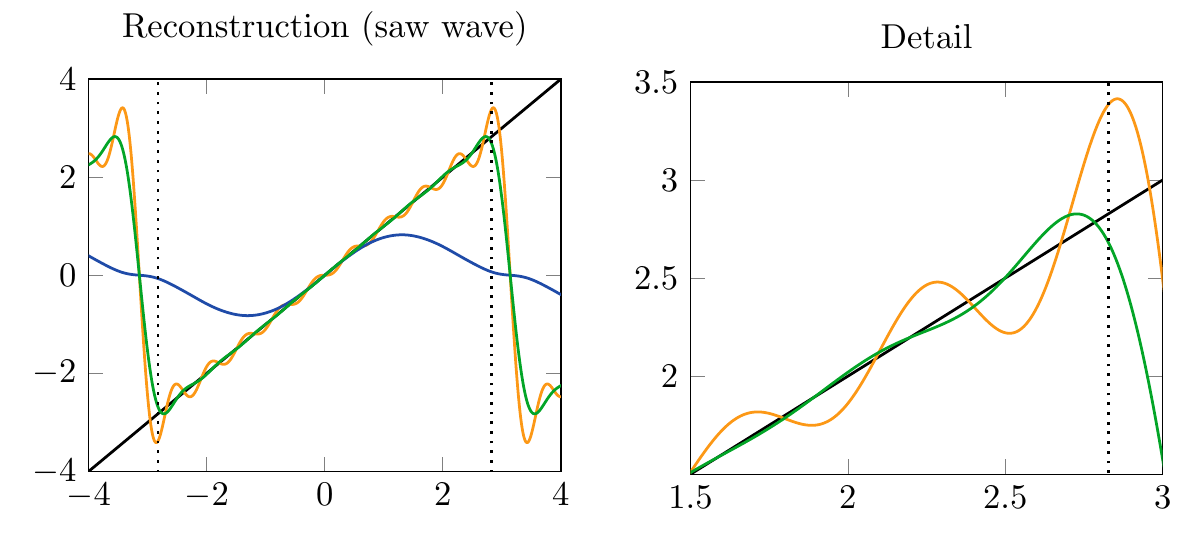}
	\caption{Plot of the reconstructions $R_{10}\psi$ and $R_{10}^w\psi$ for the saw wave function.
		For $x\in[-\rho,\rho]$ (dotted lines) the bump-windowed reconstruction (green) matches well with the original function and the typical overshoots \textcolor{black}{(Gibbs phenomenon)} of the Fourier sum (orange) are dampened.
		The reconstruction with the Hann window (blue) is accurate only in a small neighborhood of $0$.}
	\label{fig:saw_rec}
\end{figure}
%%%%%%%%%%%%%%%%%%%%%%%%%%%%%%%%%%%

In the first example we consider the function
\begin{align*}
	\psi(x)=x,\quad \Big[\lambda=\pi,\ \rho = 0.9\pi,\ t=0\Big].
\end{align*}
The corresponding periodic extension $\Per_\lambda \psi$ results in a saw wave function.\\
We note that $c_\psi(k)$ and $c_\psi^{\operatorname{hann}}(k)$ can be evaluated analytically and are given by
\begin{align*}
	c_\psi(k)
	=i\cdot\frac{(-1)^k}{k},\ k\in\Z\backslash\{0\}, \quad
	c_\psi^{\operatorname{hann}}(k)
	=-i\cdot\frac{(-1)^k}{2k(k^2-1)},\ k\in \Z\backslash\{-1,0,1\}
\end{align*}
\textcolor{black}{and $c_\psi(0)=c_\psi^{\operatorname{hann}}(0)=0,c_\psi^{\operatorname{hann}}(-1)=-c_\psi^{\operatorname{hann}}(1)=3i/8$.}
Moreover, since $\psi$ is a real function, we conclude that
\begin{align*}
	c^w_\psi(-k)=\overline{c^w_\psi(k)},\quad k\in\Z.
\end{align*}
The upper left hand side of Figure~\ref{fig:saw} shows $|c_\psi(k)|^2=1/k^2$, as well as $|c_\psi^w(k)|^2$ for both windows (hann \textcolor{black}{and} bump).
We observe that the windowed coefficients have a faster asymptotic decay than the plain Fourier coefficients.
The coefficients and the reconstruction errors for the bump (green) show the best asymptotic decay.
\textcolor{black}{As we observe in the upper right plot of Figure~\ref{fig:saw}, the bump-windowed coefficients show exponential initial decay.}
In particular, we recognize a trembling for these coefficients, while the other (plain and hann) have a smooth decay.
We provide an explanation of this phenomenon in Appendix~\ref{sec:oscillations_appendix}.
The reconstructions $R_{10}\psi$ and $R_{10}^w\psi$ are visualized in Figure~\ref{fig:saw_rec}.
\textcolor{black}{For the bump we recognize a good convergence to the original function $\psi$ in $[-\rho,\rho]$ (dotted lines), and the typical overshoots of the Fourier sum at the discontinuity (Gibbs phenomenon, see e.g. \cite[\S3]{Tadmor:2007aa}) are dampened.}
As expected, the reconstruction with the Hann window is accurate only in a small neighborhood of the center $t=0$, and according to Theorem~\ref{fact:error_rep_window} and Corollary~\ref {fact:error2_rep_window} the reconstruction errors converge to $\color{black}{K_\infty(\psi,w,\rho),K_2(\psi,w,\rho)}>0$.
For the saw wave these constants can be calculated analytically in terms of $\lambda$ and $\rho$, and their values are given by $K_\infty
	\approx 8.91$ and $K_2\approx 2.76$.
We have marked these values with red crosses.
In fact we observe a perfect match.
\begin{remark}
	\textcolor{black}{As we have discussed in \S\ref{sub:rep_bump_windows}, the coefficients of the bump do not fall exponentially fast for all $k$, since the bump is not analytic.
	However, in \cite{Boyd:2006aa} the author presents a smooth bump that is based on the erf-function, such that the Fourier coefficients for the saw wave fall exponentially fast (the exponential is of the square root of $k$).
	This is achieved by an optimization of the corresponding window parameters.
	In view of the bump used here, this relates to an optimal choice of $\rho$.
	}
\end{remark} 
%

%%%%%%%%%%%%%%%%%%%%%
%%%%%%%%%%%%%%%%%%%%%
\subsection{Parabola}\label{sub:parabola}
We consider the symmetric function 
\begin{align*}
	\psi(x)=x^2,\quad \Big[\lambda=1,\ \rho_1=0.25,\ \rho_2=0.8,\ t=0\Big].
\end{align*}
Note that
\begin{align*}
	c_\psi(k)
	=\frac{2\cdot(-1)^k}{k^2\pi^2},\ k\in\Z\backslash\{0\}, \quad
	c_\psi^{\operatorname{hann}}(k)
	=\frac{(-1)^k(1-3k^2)}{k^2(k^2-1)^2\pi^2},\ k\in \Z\backslash\{-1,0,1\},
\end{align*}
\textcolor{black}{
as well as
\begin{align*}
	c_\psi(0)
	=\frac{1}{3},\quad
	c_\psi^{\operatorname{hann}}(0)
	=\frac{1}{6}-\frac{1}{\pi^2},\quad
	c_\psi^{\operatorname{hann}}(-1)
	=c_\psi^{\operatorname{hann}}(1)
	=\frac{1}{12}-\frac{7}{8\pi^2}.
\end{align*}
}

%%%%%%%%%%%%%%%%%%%%%%%%%%%%%%%%%%%
%figure
\begin{figure}
	\includegraphics{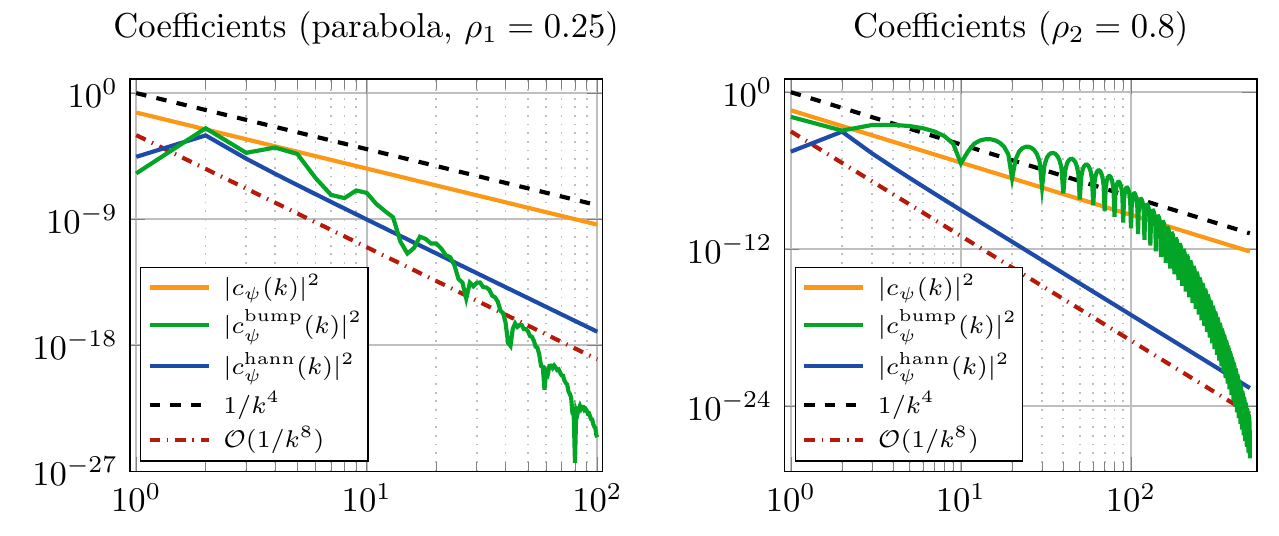}
	\caption{Decay of the representation coefficients for the parabola with $\rho_1=0.25$ (left) and $\rho_2=0.8$ (right).
	Again, the coefficients for the bump show a fast asymptotic decay.}
	\label{fig:parabola1}
\end{figure}
%%%%%%%%%%%%%%%%%%%%%%%%%%%%%%%%%%%
%figure
\begin{figure}
	\includegraphics{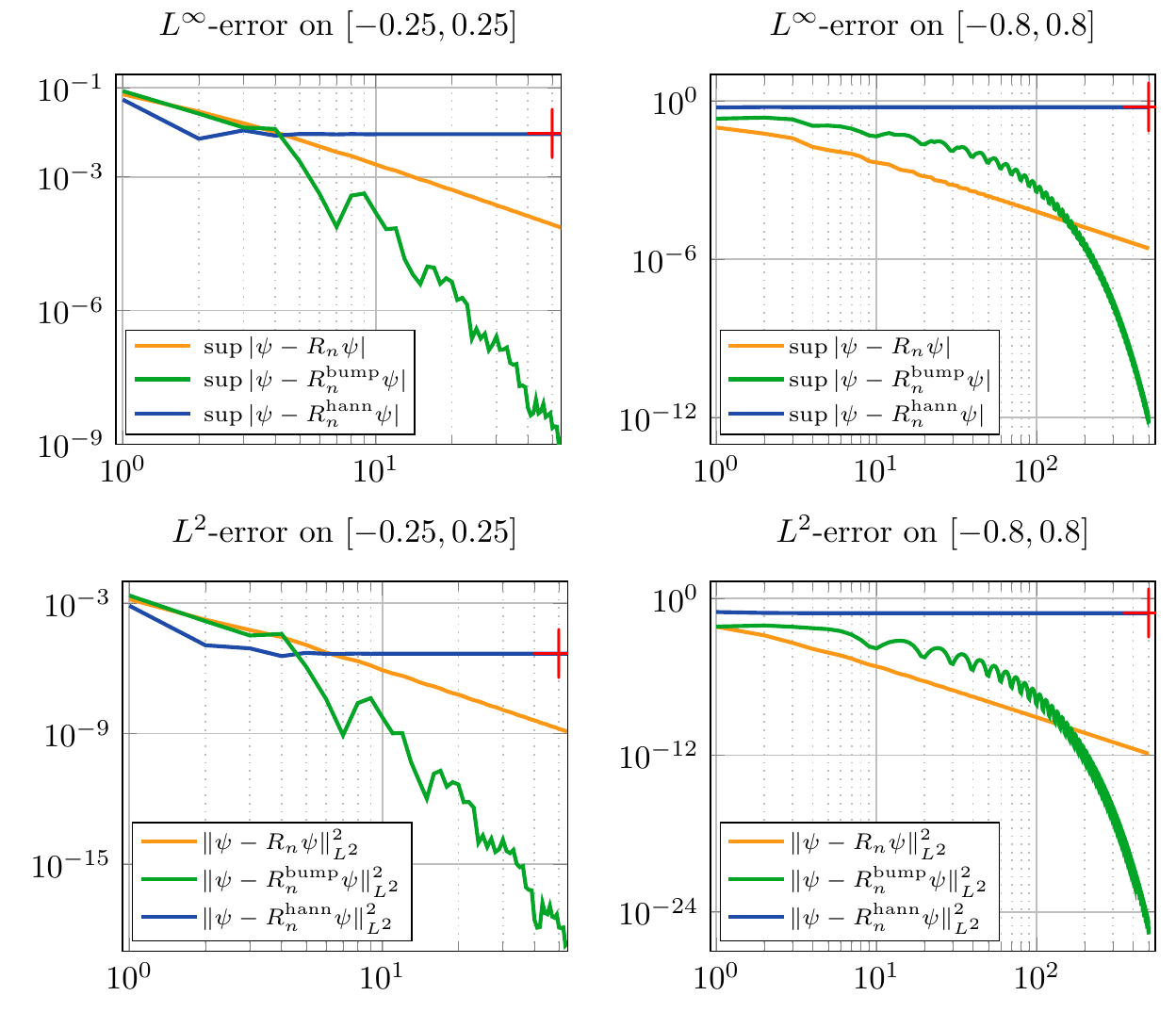}
	\caption{Reconstruction errors for the parabola with $\rho_1=0.25$ (left) and $\rho_2=0.8$ (right).
		For the second choice the smooth bump has a high derivative in the interval $(0.8,1)$, implying a large Lipschitz constant $L_s$.
		Consequently, for low frequencies the errors are worse than for the plain coefficients.}
	\label{fig:parabola2}
\end{figure}
%%%%%%%%%%%%%%%%%%%%%%%%%%%%%%%%%%%
%figure
\begin{figure}
	\includegraphics{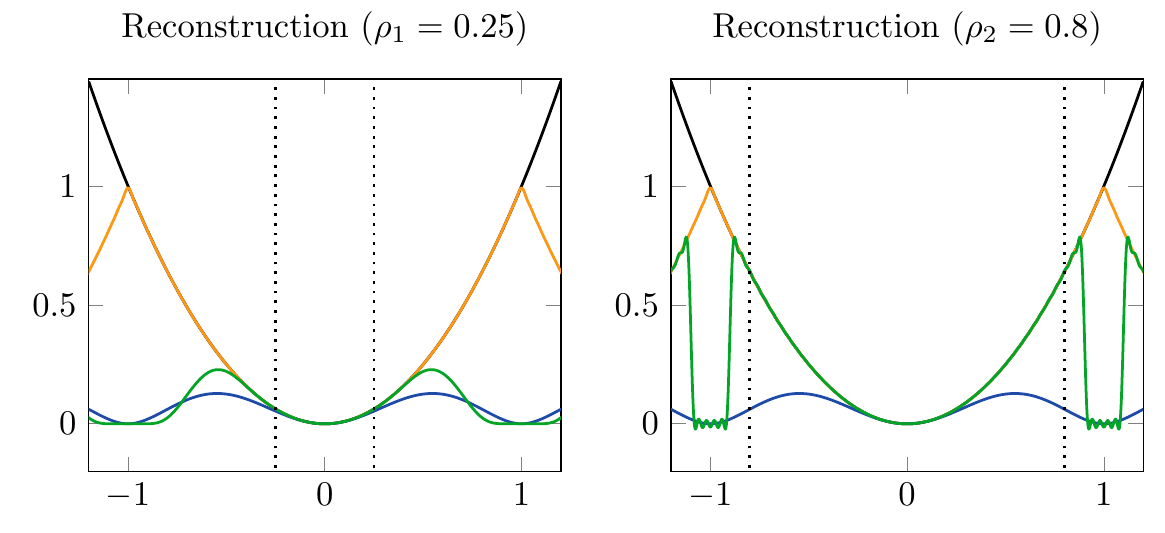}
	\caption{Reconstructions of the parabola.
		For $\rho_2=0.8$ (right) the bump-windowed shape (green) has a high derivative in $(0.8,1)$, implying a slow decay of the windowed coefficients.
		For $\rho_1=0.25$ (left) the coefficients fall off much faster, but the reconstruction is faithful only in a small interval, comparable to the Hann window.}
	\label{fig:parabola_rec}
\end{figure}
%%%%%%%%%%%%%%%%%%%%%%%%%%%%%%%%%%%

The plots in Figure~\ref{fig:parabola1} show the decay of the coefficients.
Especially for low frequencies, the coefficients for the Hann window show the fastest decay.
Nevertheless, we observe once more that the bump coefficients and errors have the best asymptotics, see Figure~\ref{fig:parabola2}.
As with the saw wave, the constants $\color{black}{K_\infty(\psi,w,\rho)}$ and $\color{black}{K_2(\psi,w,\rho)}$ can be calculated analytically and are given by
\begin{align*}
	K_\infty
	\approx
	\begin{cases}
		9.1\cdot 10^{-3},&\text{if $\rho=0.25$},\\
		0.58,& \text{if $\rho=0.8$},
	\end{cases}
	\quad\text{and}\quad
	K_2
	\approx
	\begin{cases}
		4.7\cdot10^{-6},&\text{if $\rho=0.25$},\\
		0.075,& \text{if $\rho=0.8$}.
	\end{cases}
\end{align*}
We have marked these values with red crosses and verify the predicted convergence of the errors.
The reconstructions $R_{50}(\psi)$ and $R^w_{50}(\psi)$ are visualized in Figure~\ref{fig:parabola_rec}.
For the first choice $\rho_1=0.25$ (left) the bump-windowed series approximates the original function only in the small interval $[-\rho_1,\rho_1]=[-0.25,0.25]$.
We note that the periodic extension of the parabola has no discontinuities and therefore the plain reconstruction gives a good approximation, even with few coefficients.\\

For a bad choice of the parameter $\rho$, the reconstruction with the bump gets worse.
According to Theorem~\ref{fact:estimate_lipschitz}, the Lipschitz constant $L_s$ is getting large as $\rho \to \lambda$, implying a slow decay for low frequencies, which can particularly be observed for the choice $\rho_2=0.8$. This value leads to a high derivative of the smooth bump $w_{0.8,1}$ in the interval $(0.8,1)$.
For low frequencies, the coefficients and the errors for the bump show a slow decay (right plots in Figure~\ref{fig:parabola1},\ref{fig:parabola2}) and are even worse than for the plain Fourier series.

%%%%%%%%%%%%%%%%%%%%%%%%%%%%%%%%%%%%%%%%%
%%%%%%%%%%%%%%%%%%%%%%%%%%%%%%%%%%%%%%%%%
\subsection{A function of rapid decrease}\label{sub:hermite}
%%%%%%%%%%%%%%%%%%%%%%%%%%%%%%%%%%%
%figure
\begin{figure}
	\includegraphics{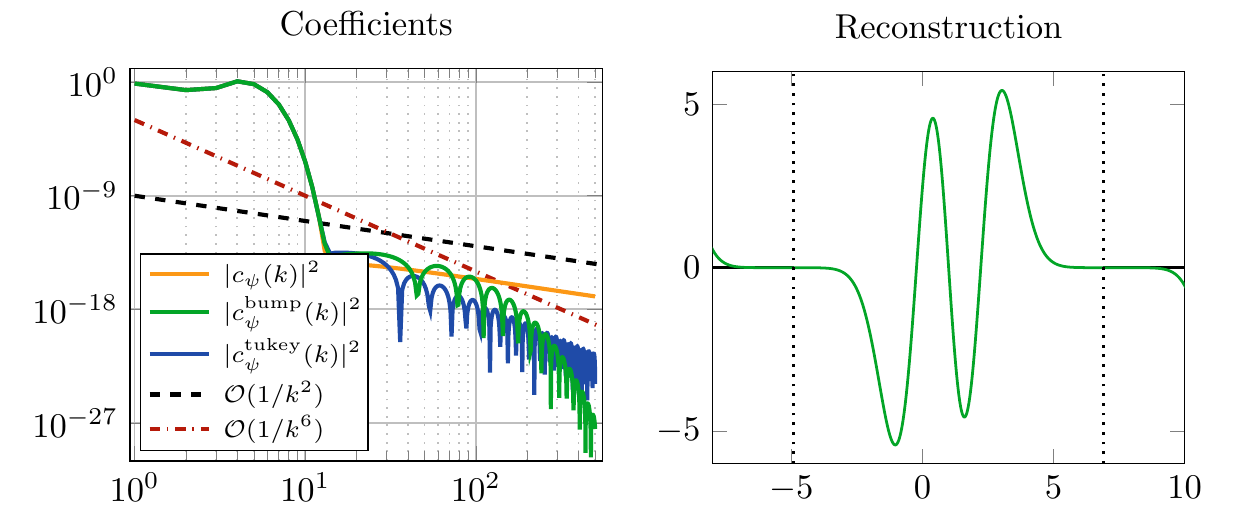}
	\caption{Coefficients (left) and reconstructions $R_{10}\psi$ and $R_{10}^w\psi$ (right) for the rescaled Hermite function.
		All coefficients show a rapid decrease for low frequencies and thus we obtain excellent reconstructions for all series.}
	\label{fig:rectangular}
\end{figure}
%%%%%%%%%%%%%%%%%%%%%%%%%%%%%%%%%%%

We also applied the transforms to
\begin{align*}
	\psi(x)
	=\left(8x^3-24x^2+12x+4\right)e^{-(x-1)^2/2},\quad \Big[\lambda=2\pi,\ \rho = 5.9,\ t=1\Big].
\end{align*}
We note that $\psi(x+1)$ is the product of the Hermite polynomial $H_3(x)=8x^3-12x$ times a Gaussian, i.e. a rescaled Hermite function.
For the center we chose $t=1$.
In contrast to the previous examples, we now work with the Tukey window for $\alpha=1-\rho/\lambda$, see Definition~\ref{def:tukey_window}.
We recall that this window is a \textcolor{black}{non-degenerate} $C^1$-bump.
The $2\lambda$-periodic extension of $\psi$ produces discontinuities with very small jumps, which can only be resolved with high frequencies.
Consequently, for low frequencies all coefficients are almost the same and fall off rapidly, see Figure~\ref{fig:rectangular}.
Nevertheless, the plain coefficients are $\mathcal{O}(1/|k|)$, while the coefficients for the smooth bump again show the best asymptotic decay.
For the reconstructions we used $R_{10}\psi$ and $R_{10}^w\psi$.
As we observe in the right plot of Figure~\ref{fig:rectangular}, the rapid decrease of the coefficients yields excellent reconstructions and no differences can be determined to the original function.
\appendix
%%%%%%%%%%%%%%%%%%%%%%%
%%%%%%%%%%%%%%%%%%%%%%%
%%%%%%%%%%%%%%%%%%%%%%%
\section{Periodization}\label{sec:periodization_appendix}
\begin{lemma}\label{app:fact_periodization}
	For $\lambda>0$ and $\psi\in \BV_{\emph{loc}}$ we have $\Per_\lambda \psi\in \BV_{\emph{loc}}$.
\end{lemma}
\begin{proof}
	For a function $f\colon[a,b]\to \R$ and a partition $P$ of some finite interval $[a,b]$ we denote by $V(f,P)$ the variation of $f$ with respect to $P$, and by $V(f)$ the total variation of $f$ on $[a,b]$.
	Now, for $\psi\in \BV_{\text{loc}}$ and $\lambda>0$ consider $f:=\Per_\lambda \psi$.
	It remains to show that $V(f|_{[-\lambda,\lambda]})$ is a finite number.
	Therefore, let 
	\begin{align*}
		P=\{-\lambda=x_0,x_1,...,x_{k-1},x_k=\lambda\}
	\end{align*}
	be a partition of $[-\lambda,\lambda]$.
	Then,
	\begin{align*}
			V(f,P)
			&=\sum_{i=1}^k|f(x_i)-f(x_{i-1})|\\
			&\le\sum_{i=1}^k|\psi(x_i)-\psi(x_{i-1})|
			+|\psi(-\lambda)-f(-\lambda)|+|f(\lambda)-\psi(\lambda)|\\[2mm]
			&=V(\psi,P)+|\psi(-\lambda)-f(-\lambda)|+|f(\lambda)-\psi(\lambda)|.
	\end{align*}
	Thus, taking the supremum among such partitions, we conclude that
	\begin{align*}
		V(f|_{[-\lambda,\lambda]})
		=V(\psi|_{[-\lambda,\lambda]})+|\psi(-\lambda)-f(-\lambda)|+|f(\lambda)-\psi(\lambda)|<\infty.
	\end{align*}
\end{proof}
%

%%%%%%%%%%%%%%%%%%%%%%%%%%%%%%%
%%%%%%%%%%%%%%%%%%%%%%%%%%%%%%%
%%%%%%%%%%%%%%%%%%%%%%%%%%%%%%%
\section{Upper bound for $K_s$}\label{sec:upper_bound_appendix}
Recall the representation of the combinatorial constant $K_s$ in \eqref{eq:rep_ks}.
We want to find an estimate for the following sum, cf. equation \eqref{eq:rep_ks_sum}:
\begin{align*}
	\sum_{k=0}^{s-1} \frac{(s+k)!\cdot(2s-k-1)!}{(2k+2)!\cdot(s-k-1)!\cdot(2s-2k)!\cdot k!}
	=\frac{K_s}{2^{2s}\cdot (s+1)!\cdot (2s)^2}.
\end{align*}
For the summand we calculate
\begin{align}\label{eq:calc_ks1}
	\begin{split}
		&\frac{(s+k)!\cdot(2s-k-1)!}{(2k+2)!\cdot(s-k-1)!\cdot(2s-2k)!\cdot k!}\\
		=\ &\frac{(s+k)!\cdot(2s-k-1)!}{s!\cdot k!\cdot(s-k-1)!\cdot s!}\cdot\frac{s!\cdot s!}{(2s+2)!}\cdot\binom{2s+2}{2k+2}.
	\end{split}
\end{align}
Recall Vandermonde's theorem, see e.g. \cite[Equation (2.43)]{Seaborn:1991aa}:
\begin{align*}
	\sum_{k=0}^m \frac{(a)_k}{k!}\frac{(b)_{m-k}}{(m-k)!}
	=\frac{(a+b)_m}{m!},
	\quad a,b\in\C,\ m\ge 0.
\end{align*}
In particular, for $m=s-1$ and $a=b=s+1,\ s\ge 1$, we obtain
\begin{align}\label{eq:calc_ks2}
	\sum_{k=0}^{s-1} 
	\frac{(s+k)!\cdot(2s-k-1)!}{s!\cdot k!\cdot(s-k-1)!\cdot s!}
	=\binom{3s}{s-1}.
\end{align}
Hence, since
\begin{align*}
	\binom{2s+2}{2s-2k}
	\le \binom{2s+2}{s}
	\quad\text{for}\quad 
	0\le k\le s-1,\ s\ge 2,
\end{align*}
by \eqref{eq:calc_ks1} and \eqref{eq:calc_ks2} we conclude that 
\begin{align*}
	\sum_{k=0}^{s-1} \frac{(s+k)!\cdot(2s-k-1)!}{(2k+2)!\cdot(s-k-1)!\cdot(2s-2k)!\cdot k!}
	&\le\frac{s!\cdot s!}{(2s+2)!}\binom{2s+2}{s}\binom{3s}{s-1}\\
	&=\frac{(3s)!}{s!\cdot(2s+2)!}\cdot \frac{2s}{s+2}.
\end{align*}
This proves that
\begin{align*}
	K_s
	\le \frac{2^{2s+1}\cdot s^2\cdot (3s)!}{ (2s+1)!}\cdot \frac{2s}{s+2},
	\quad s\ge 2.
\end{align*}
Consequently, the true value of $K_s$ is overestimated by the factor $2s/(s+2)$.

%%%%%%%%%%%%%%%%%%%%%%%%%%%%%%%%%%%%%%%%%%%%%%%%%%%
%%%%%%%%%%%%%%%%%%%%%%%%%%%%%%%%%%%%%%%%%%%%%%%%%%%
%%%%%%%%%%%%%%%%%%%%%%%%%%%%%%%%%%%%%%%%%%%%%%%%%%%
\section{Computing Fourier Integrals using the FFT}\label{sec:fft_appendix}
In \S\ref{sec:numeric} we presented numerical results for reconstructions based on windowed Fourier coefficients and windowed series, respectively.
For the computation of Fourier-type integrals, such as
\begin{align}\label{fourier:type:integral}
	I^{(k)}(\psi)
	:=\int_{t-\lambda}^{t+\lambda} \psi(x)w(x-t)e^{-i\xi\frac{\pi}{\lambda}x}\ \mathrm{d}x,\quad \xi=\frac{k}{M},\ k\in\Z,\ M\in\N,
\end{align}
we used the fast Fourier transform (FFT).
Let $v = (v_1,...,v_d)\in\C^d$ and
\begin{align}\label{dft}	
	\hat v_l
	:=\sum_{j=1}^d v_{j}e^{-2\pi i \cdot (j-1)(l-1)/d},\quad l\in\{1,...,d\}.
\end{align}
For the computation of the integral $I^{(k)}(\psi)$ let us consider the \textcolor{black}{composite} trapezoidal rule on a uniform grid.
Let $N\in\N$ be a power of 2, as well as $m\in \N,\ m\le N-1$ and
\begin{align*}
	x_j
	:=t-\lambda+\Delta\cdot j,\quad 
	\Delta
	:= \frac{2\lambda}{m},\quad 
	j\in\{0,1,...,m\}.
\end{align*} 
The trapezoidal rule with grid $\{x_j\}_{j\in\{0,1,...,m\}}$ yields the following approximation:
\begin{align*}
	I^{(k)}(\psi)
	\approx e^{-i\frac{k}{M}\cdot\frac{\pi}{\lambda}t}e^{i\frac{k}{M}\pi}\Delta\Bigg(\frac{\psi(t+\lambda)w(\lambda)}{2}e^{-2\pi i\frac{k}{M}}-\frac{\psi(t-\lambda)w(-\lambda)}{2}\dots\\
		+\sum_{j=1}^m\psi(x_{j-1})w(x_{j-1}-t)e^{-2\pi i\frac{k}{Mm}(j-1)}\Bigg).
\end{align*}
Consequently, if $k=mn$ for some $n\in \{0,1,...,N-1\}$, as well as $M=N$ and $v_j:=\psi(x_{j-1})w(x_{j-1}-t)$ for $j=1,...,m,\ v_j:=0$ for $j=m+1,...,N$, then,
\begin{align*}
	I^{(k)}(\psi)
	&\approx e^{-i\frac{mn}{N}\cdot\frac{\pi}{\lambda}t}e^{i\frac{mn}{N}\pi}\Delta\left(r_1e^{-2\pi i\frac{mn}{N}}-r_2+\sum_{j=1}^N v_je^{-2\pi i(j-1)n/N}\right)\\
	&=e^{-i\frac{mn}{N}\cdot\frac{\pi}{\lambda}t}e^{i\frac{mn}{N}\pi}\Delta\Big(r_1e^{-2\pi i\frac{mn}{N}}-r_2+\hat v_{n+1}\Big),
\end{align*}
where the constants $r_1,r_2\in\R$ are given by
\begin{align*}
	r_1
	=\frac{\psi(t+\lambda)w(\lambda)}{2}
	\quad\text{and}\quad
	r_2
	=\frac{\psi(t-\lambda)w(-\lambda)}{2}.
\end{align*}
In particular, the vector $\hat v$ can be calculated with the FFT.
For sufficiently large values of $m$ and $N$ we get
\begin{align*}
	\frac{1}{2\lambda}\int_{t-\lambda}^{t+\lambda} \psi(x)w(x-t)e^{-i\xi\frac{\pi}{\lambda}x}\ \mathrm{d}x
	\approx \frac{e^{-i\xi\frac{\pi}{\lambda}t}e^{i\xi\pi}}{m}\left(r_1e^{-2\pi i\xi}-r_2+\hat v_{n+1}\right).
\end{align*}
\textcolor{black}{Recall that the window $w$ is compactly supported.
Provided that both the functions $\psi$ and $w$ are smooth on $(t-\lambda,t+\lambda)$, the trapezoidal rule gives accurate results.
The actual rates of convergence are based on the Euler-Maclaurin formula and can be found e.g. in \cite[chapter 2.9]{Davis:1984aa}.
In particular, the difference between the exact Fourier coefficients and their discrete approximation using the trapezoidal rule is known to be spectrally small, see \cite[Equation (1.5)]{Gottlieb:1985aa}.
}

%%%%%%%%%%%%%%%%%%%%%%%%%%%%%%%%%%%%%%%%%%
%%%%%%%%%%%%%%%%%%%%%%%%%%%%%%%%%%%%%%%%%%
%%%%%%%%%%%%%%%%%%%%%%%%%%%%%%%%%%%%%%%%%%
\section{Oscillations of the coefficients}\label{sec:oscillations_appendix}
%%%%%%%%%%%%%%%%%%%%%%%%%%%%%%%%%%%
%figure
\begin{figure}
	\includegraphics{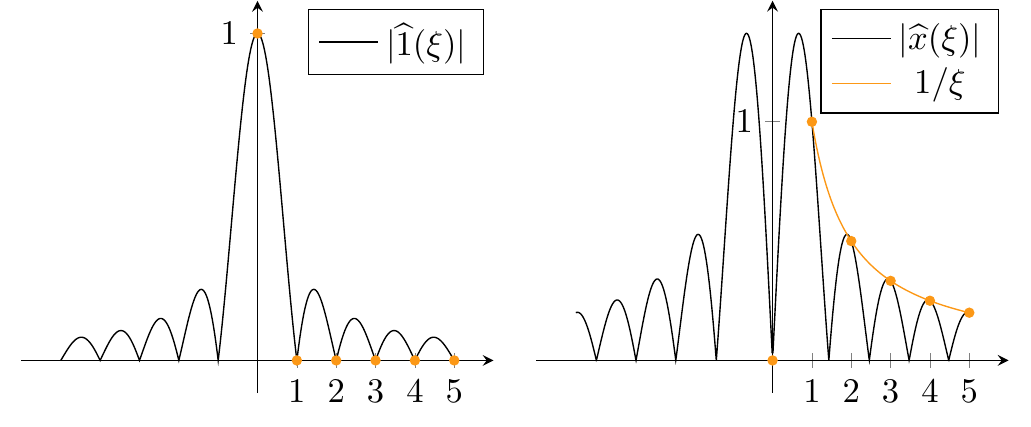}
	\caption{Absolute values of the Fourier coefficients (bullets) for $f(x)=1$ (left) and $f(x)=x$ (right).
		The extension of the domain of the Fourier coefficients leads to a non-trivial function in $\xi$, but the restriction to integer values might result in a smooth decay (orange line).}
	\label{fig:sinc}
\end{figure}
%%%%%%%%%%%%%%%%%%%%%%%%%%%%%%%%%%%
%figure
\begin{figure}
	\includegraphics{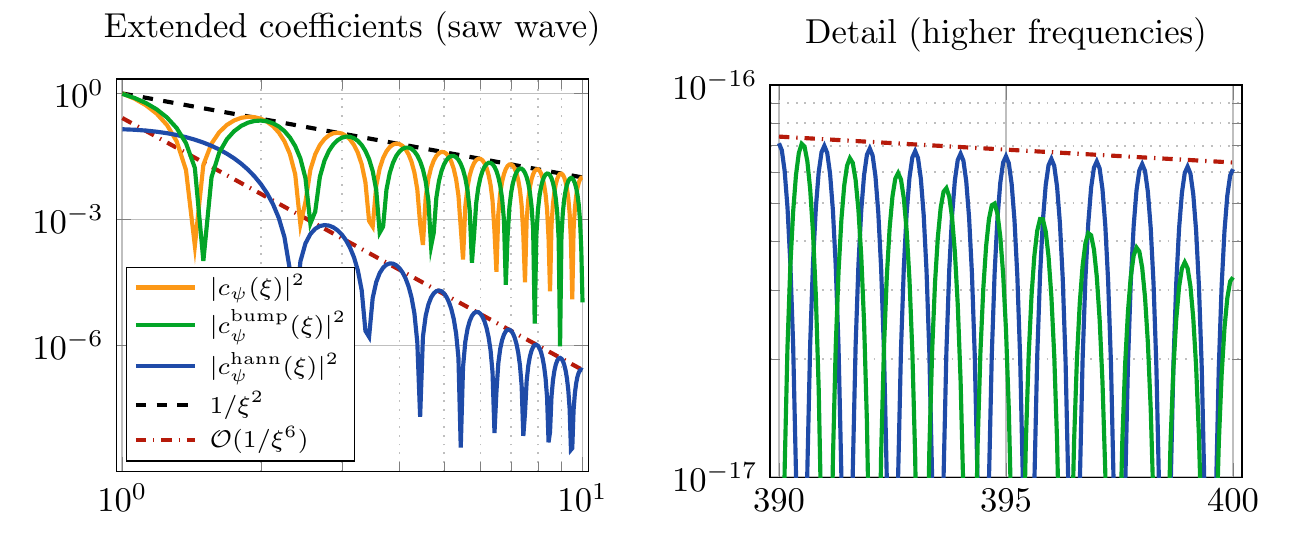}
	\caption{Due to the extension of the domain of the coefficients, we are able to resolve the pattern (cf. upper left plot in Figure~\ref{fig:saw}).
		We recall that the bump-windowed coefficients (green) have a fast asymptotic decay.
		The left plot only shows low frequencies $\xi$, and in the right plot we observe that the bump coefficients fall below the Hann coefficients.}
	\label{fig:saw_generalized}
\end{figure}
%%%%%%%%%%%%%%%%%%%%%%%%%%%%%%%%%%%

We focus once more on the windowed coefficients $c_\psi^w$.
In the plot at the upper left hand side of Figure~\ref{fig:saw} the green line \textcolor{black}{falls in a trembling way.}
To explain this phenomenon, we extend the domain of the Fourier coefficients.
For a $2\pi$-periodic function $f\in\BV_{\text{loc}}$ and $\xi\in\R$ consider the number
\begin{align*}
	\widehat f(\xi)
	:=\frac{1}{2\pi}\int_{-\pi}^\pi f(x)e^{-i\xi x}\ \mathrm{d}x.
\end{align*}
This means, that we calculate the Fourier coefficients not only for integer values, but for all real numbers $\xi$.
For example, the extended Fourier coefficients of the function $f\equiv 1$ are given by
\begin{align*}
	\widehat 1(\xi)
	:=\frac{1}{2\pi}\int_{-\pi}^\pi e^{-i\xi x}\ \mathrm{d}x
	=\frac{\sin(\pi\xi)}{\pi\xi}
	=\operatorname{sinc(\xi)}.
\end{align*}
In particular, if $k$ is an integer, we obtain the simple Fourier coefficients and
\begin{align*}
	|\widehat 1(k)|=
	\begin{cases}
		1,& \text{if}\ k=0,\\
		0,& \text{else}.
	\end{cases}
\end{align*}
As we see in the left plot of Figure~\ref{fig:sinc}, for $k\ne 0$ the simple Fourier coefficients of $f\equiv 1$ correspond to the zeros of $\xi\mapsto|\operatorname{sinc(\xi)}|$.
For the saw wave in \S\ref{sub:saw_wave} we can do the same calculation.
Here we obtain
\begin{align*}
	\widehat x(\xi)
	:=\frac{1}{2\pi}\int_{-\pi}^\pi xe^{-i\xi x}\ \mathrm{d}x
	=i\cdot\frac{(\pi\xi\cos(\pi\xi)-\sin(\pi\xi))}{\pi\xi^2}.
\end{align*}
Therefore, if $k\ne 0$ is an integer, we conclude that
\begin{align*}
	|\widehat x(k)|
	=\frac{1}{k}.
\end{align*}
Thus, the coefficients of the saw wave function have a smooth decay, as we see at the right hand side of Figure~\ref{fig:sinc} (orange line).
We computed the extended (windowed) coefficients for $\xi\in[1,10]$ and $\xi\in[390,400]$ for the saw wave.
The result can be found in Figure~\ref{fig:saw_generalized}.
By extending the domain of the Fourier coefficients, we observe that the trembling also occurs for the other coefficients (plain and hann).

%\bibliographystyle{alpha}
%\bibliography{bibliography}

\end{document}